\documentclass[onefignum,onetabnum]{siamart190516}

\usepackage{amsmath}

\usepackage{amsfonts, psfrag, graphicx, amssymb,enumitem,indentfirst,float,geometry,color,enumerate,graphicx,subcaption,algorithm,algpseudocode,pifont,hyperref,mathtools,mathrsfs}
\graphicspath{{./Figure/}}
\usepackage{mathbbol}
\usepackage{pbox}
\usepackage{booktabs, cellspace, hhline}
\allowdisplaybreaks[4]
\numberwithin{equation}{section}
\numberwithin{figure}{section}

\usepackage{etoolbox}
\patchcmd{\thebibliography}{\chapter*}{\section*}{}{}
\usepackage[title]{appendix}

\usepackage{lipsum}

\newtheorem{assumption}{Assumption}
\newsiamremark{remark}{Remark}

\newcommand{\commentout}[1]{{}} 

\newcommand{\abs}[1]{\left|#1\right|}

\newcommand{\bfa}{{\bf a}}
\newcommand{\bfB}{{\bf B}}
\newcommand{\bfb}{{\bf b}}

\newcommand{\bfe}{{\bf e}}

\newcommand{\bff}{{\bf f}}

\newcommand{\bfH}{{\bf H}}

\newcommand{\bfK}{{\bf K}}

\newcommand{\bfn}{{\bf n}}

\newcommand{\bfr}{{\bf r}}

\newcommand{\bft}{{\bf t}}

\newcommand{\bfu}{{\bf u}}
\newcommand{\bfV}{{\bf V}}
\newcommand{\bfv}{{\bf v}}
\newcommand{\bfw}{{\bf w}}

\newcommand{\bfx}{{\bf x}}
\newcommand{\bfX}{{\bf X}}

\newcommand{\bfy}{{\bf y}}
\newcommand{\bfY}{{\bf Y}}

\newcommand{\bfphi}{\boldsymbol{\phi}}
\newcommand{\bfvarphi}{\boldsymbol{\varphi}}

\newcommand{\bfpsi}{\boldsymbol{\psi}}

\newcommand{\bfTheta}{\boldsymbol{\Theta}}
\newcommand{\bfxi}{\boldsymbol{\xi}}
\newcommand{\bfeta}{\boldsymbol{\eta}}

\newcommand{\ddiv}{\operatorname{div}}
\newcommand{\rot}{\operatorname{rot}}
\newcommand{\brot}{\boldsymbol{\operatorname{rot}}}
\newcommand{\curl}{\operatorname{curl}}

\newcommand{\dist}{\operatorname{dist}}

\newcommand{\dd}{\,{\rm d}}

\newcommand{\vertiii}[1]{{\left\vert\kern-0.25ex\left\vert\kern-0.25ex\left\vert #1
    \right\vert\kern-0.25ex\right\vert\kern-0.25ex\right\vert}}
    \newcommand{\vertii}[1]{{\left\vert\kern-0.25ex\left\vert #1
    \right\vert\kern-0.25ex\right\vert}}

\begin{document}
\title{
Anisotropic analysis of VEM for time-harmonic Maxwell equations in inhomogeneous media
with low regularity
}
\author{
Chunyu Chen 
\thanks{School of Mathematics and Computational Science, Xiangtan University, National Center for
Applied Mathematics in Hunan, Hunan Key Laboratory for Computation and Simulation in Science
and Engineering, Xiangtan 411105, China (202131510114@smail.xtu.edu.cn).}
\and Ruchi Guo \thanks{Department of Mathematics, University of California, Irvine, CA 92697 (ruchig@uci.edu)}
\and Huayi Wei \thanks{School of Mathematics and Computational Science, Xiangtan University, National Center for
Applied Mathematics in Hunan, Hunan Key Laboratory for Computation and Simulation in Science
and Engineering, Xiangtan 411105, China (weihuayi@xtu.edu.cn).}
\funding{The first and third authors were supported by Huawei Co. Ltd.(Grant No. TC20220506026) and the National Natural Science Foundation of China (NSFC) (Grant No. 12261131501, 11871413) and the construction of innovative provinces in Hunan Province (Grant No. 2021GK1010).
}}
\date{}
\maketitle
\begin{abstract}
It has been extensively studied in the literature that solving Maxwell equations is very sensitive to the mesh structure, space conformity and solution regularity.
Roughly speaking, for almost all the methods in the literature, optimal convergence for low-regularity solutions heavily relies on conforming spaces and highly-regular simplicial meshes.
This can be a significant limitation for many popular methods based on polytopal meshes in the case of inhomogeneous media, 
as the discontinuity of electromagnetic parameters can lead to quite low regularity of solutions near media interfaces, 
and potentially worsened by geometric singularities, 
making many popular methods based on broken spaces, non-conforming or polytopal meshes particularly challenging to apply. 
In this article, we present a virtual element method for solving an indefinite time-harmonic Maxwell equation in 2D inhomogeneous media with quite arbitrary polytopal meshes,
and the media interface is allowed to have geometric singularity to cause low regularity. 
There are two key novelties: 
(i)  the proposed method is theoretically guaranteed to achieve robust optimal convergence for solutions with merely $\bfH^{\theta}$ regularity, $\theta\in(1/2,1]$; 
(ii) the polytopal element shape can be highly anisotropic and shrinking, and an explicit formula is established to describe the relationship between the shape regularity and solution regularity. 
Extensive numerical experiments will be given to demonstrate the effectiveness of the proposed method.
\end{abstract}


\section{Introduction}

The time-harmonic Maxwell equation has numerous applications, such as the design and analysis of various electromagnetic devices, non-destructive detection techniques fields and so on~\cite{2015AmmariChenChenVolkov,2020CHENLIANGZOU}. 
These applications often incorporate multiple media with distinct electromagnetic parameters referred to as inhomogeneous media.
Conventional finite element methods (FEMs) have been extensively studied for this type of equations \cite{2023ChenLiXiang,2014CiarletWuZou,1992Monk,2003GopalakrishnanPasciak,2009ZhongShuWittumXu}, which often require a simplicial mesh that conforms to the interface and meets demanding shape regularity conditions to achieve optimal accuracy. 
However, generating such a mesh for complex media distributions with intricate interface geometry can be highly challenging. 

\subsection{The time-harmonic Maxwell equation and its discretization}

Recall the following 2D time-harmonic Maxwell equation:
\begin{align}
\label{model0}
  \brot \mu^{-1} \rot \bfu - \omega^2(\epsilon + i \sigma/\omega) \bfu = \bff ~~~~ \text{in} ~ \Omega \subseteq \mathbb{R}^2,
\end{align}
with $\bff \in \bfH(\ddiv;\Omega)$ being the source term,
subject to, say, the homogeneous boundary conditions. 
Here, $\mu>0$, $\omega> 0$, $\epsilon \ge 0$ and $\sigma \ge 0$ denote the magnetic permeability, angular frequency, electric permittivity and conductivity, respectively, and $i = \sqrt{-1}$. 
Given $\bfv = [v_1,v_2]^T$ being a 2D field (column vector functions) and $v$ being a scalar function,
let $\rot \bfv = \partial_{x_1} v_2 - \partial_{x_2} v_1$ and $\brot v = [\partial_{x_2}v , -\partial_{x_1}v]^T$ denote the rotational operators.
The motivation for this work stems from the practical applications of electromagnetic waves propagating through inhomogeneous or heterogeneous media, 
where $\mu$, $\epsilon$ and $\sigma$ are discontinuous constants that represent the distinct electromagnetic properties of various media.
We refer readers to Figures \ref{fig:example12}, \ref{fig:interfacedoublecircle} and \ref{fig:twolayerspolymesh} for examples of inhomogeneous media distribution.

Henceforth, we rewrite \eqref{model0} into 
\begin{equation}
\label{model1}
 \brot \alpha \rot \bfu - \beta \bfu = \bff ~~~~ \text{in} ~ \Omega,
\end{equation}
and the corresponding variational form is
\begin{equation}
\label{model2}
a(\bfu,\bfv) := (\rot \bfu, \rot \bfv)_{\Omega} - (\beta \bfu, \bfv)_{\Omega} = (\bff, \bfv), ~~~~ \forall \bfv\in \bfH_0(\rot;\Omega).
\end{equation}
If $\sigma\neq 0$, the bilinear form in \eqref{model2} is coercive on $\bfH(\rot;\Omega)$ \cite{1992Monk,2003GopalakrishnanPasciak} (also see a proof in \cite[Lemma 2.5]{2019BonazzoliDoleanGrahamSpenceTournier}), 
and in such a case the analysis follows from the standard argument of Lax-Milgram lemma. 
Note that the well-posedness of \eqref{model1} implicitly implies the following conditions:
\begin{equation}
\label{interfaceCond}
[\bfu\cdot\bft]_{\Gamma}  = 0 ~~~~ \text{and} ~~~~  [\alpha \rot \bfu]_{\Gamma} = 0,
\end{equation}
and $\bff \in \bfH(\ddiv;\Omega)$ also implies $[\beta \bfu \cdot\bfn]_{\Gamma} = 0$ which, however, will not be explicitly used.
In this work, we restrict ourselves to the case of $\sigma=0$ and $\beta = \omega^2 \epsilon >0$ being not eigenvalues of the relevant system, and thus \eqref{model2} is indefinite posing challenges for either the computation and analysis.

The relaxation of element shape can largely facilitate the mesh generation procedure, 
which has been extensively studied in the past decades. 
For existing works on solving various equations polytopal meshes, 
we refer readers to discontinuous Galerkin (DG) methods \cite{2013AntoniettiGianiHouston,2017CangianiDongGeorgoulisHouston,2014CangianiGeorgoulisHouston,2018ZhaoPark}, 
hybrid DG methods \cite{2020PietroDroniou,2021DuSayas}, 
mimetic finite difference methods \cite{2005BrezziLipnikovShashkov,2005BREZZILIPNIKOVSIMONCINI,2010VeigaManzini}, 
weak Galerkin methods \cite{2012MuWwangYe,2014WangYe}, 
and virtual element methods (VEMs) \cite{2017VeigaCarloAlessandro,2013BeiraodeVeigaBrezziCangiani,2018BrennerSung,2021CaoChenGuoIVEM,2022CaoChenGuo} to be discussed in this work. 
Particularly, it has been studied in \cite{2017ChenQiuShiSolana,2020DuSayasSIAM,2008HermelineLayouniOmnes,2015KretzschmarMoiolaPerugia,2015ChenWang,2004CockburnLiShu} for solving Maxwell equations on general polytopal meshes. 
Since it is not very feasible to construct conforming polynomial spaces on general polytopal meshes, almost all the aforementioned DG-type works typically all employ broken polynomial spaces for approximation.
We also refer readers to \cite{2023FallettaFerrariScuderi,2021MengWangMei,2016PerugiaPietraRusso} for VEMs on solving various time-harmonic equations.

The method in the present work offers numerous advantages in various scenarios, including but not limited to: 
(i) semi-structure meshes that are cut by arbitrary interface from background Cartesian meshes \cite{2021CaoChenGuo,2017ChenWeiWen}, see the left two plots in Figure \ref{fig:NonConfMesh};  
(ii) a non-conforming mesh in which one mesh slids over another one for moving objects, a situation where Mortar FEMs are widely used \cite{2000BuffaMadayRapetti,2010LangeHenrotteHameyer,2008HuShuZou,2001BenBuffaMaday}, 
see right two plots in Figure \ref{fig:NonConfMesh};
(iii) refined anisotropic meshes adaptive to the singularity \cite{2002Nicaise}.
These scenarios frequently occur during electromagnetic phenomenon simulations, 
see \cite{1999ALONSO,2016BonitoGuermondLuddens,2016PatrickMarioBarbara,2018LanteriParedesScheid}.
However, to our best knowledge, the aforementioned works based on semi-structured or polygonal meshes demand two critical conditions: 
(a) the solutions should be sufficiently smooth admitting at least the $H^2$ regularity; and (b) the polygonal element is shape regular in the sense of star convexity and containing no short edges/faces. 

This very issue at hand is precisely attributed to that the broken spaces used in these works inevitably require a stabilization term with an $\mathcal{O}(h^{-1})$ scaling weight to impose tangential continuity. 
As a consequence, the widely-used argument of straightforwardly applying trace inequalities can merely yield the suboptimal convergence rate $h^{\theta-1}$ for solutions with a $\bfH^{\theta}$ regularity \cite{2016CasagrandeWinkelmannHiptmairOstrowski,2016CasagrandeHiptmairOstrowski}. 
It results in the failure of convergence even for a moderate case $\theta=1$. 
Nevertheless, for Maxwell equations involving inhomogeneous and heterogeneous media, the solution regularity is well-known to be quite low, merely $\bfH^{\theta}$, $\theta\in(0,1]$, near the interface between distinct media \cite{2016Ciarlet,2004CostabelDaugeNicaise,2018Alberti,1999MartinMoniqueSerge,2000CostabelDauge}.

\begin{figure}[h]
\centering
\begin{subfigure}{.2\textwidth}
    \includegraphics[width=1in]{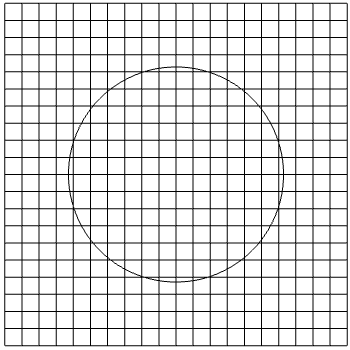}
     \label{fig:UnfitMesh1} 
\end{subfigure}
 \begin{subfigure}{.2\textwidth}
    \includegraphics[width=1.1in]{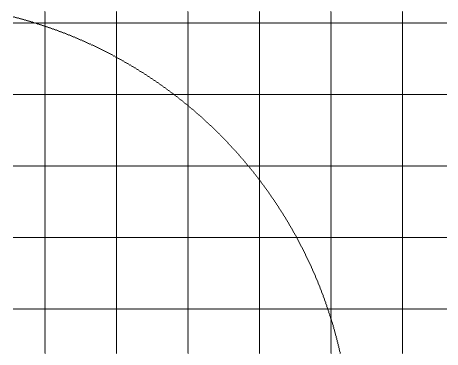}
     \label{fig:UnfitMesh2} 
\end{subfigure}
\begin{subfigure}{.2\textwidth}
    \includegraphics[width=1in]{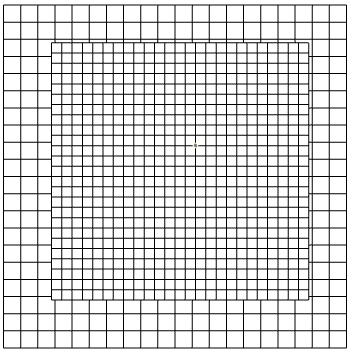}
     \label{fig:NonConfMesh1} 
\end{subfigure}
 \begin{subfigure}{.2\textwidth}
    \includegraphics[width=1in]{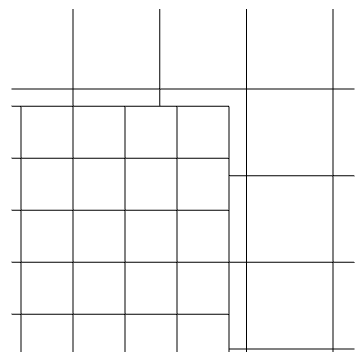}
     \label{fig:NonConfMesh2} 
\end{subfigure}
     \caption{Left two (a semi-structured mesh): a Cartesian mesh cut by a circular interface, and some elements are cut into two polygons for approximating the interface.
     Right two (a non-conforming mesh): a square interface cuts the domain into two subdomains which are discretized by two different meshes, 
     and these two meshes may not be conforming at the interface such that very complicated shaped polygon may appear around the interface.}
  \label{fig:NonConfMesh} 
\end{figure}

In fact, in \cite{2001BenBuffaMaday}, Ben Belgacem, Buffa and Maday established the first analysis for Maxwell equations with Mortar FEMs on non-conforming meshes and explained the aforementioned issue of higher regularity. The regularity was then reduced in \cite{2008HuShuZou,2000ChenDuZou} by imposing certain conforming conditions on meshes. 
Some unfitted mesh methods also employ non-conforming trial function spaces and conforming test functions spaces to achieve optimal convergence order for the $\bfH(\curl)$-type problems, see \cite{2023ChenGuoZou,2020GuoLinZou}.
High-order unfitted mesh methods can be found in \cite{2023ChenLiXiang,2023LiLiuYang,2020LiuZhangZhangZheng} for Maxwell-type problems.
As for DG or penalty type methods, the loss of convergence order has been observed in \cite{2016CasagrandeWinkelmannHiptmairOstrowski,2016CasagrandeHiptmairOstrowski} for the case of low-regularity.
Meanwhile, it is worthwhile to mention that it is possible for DG-type methods to recover optimal convergence but only on simplicial meshes. The key technique is to seek for a $\bfH(\rot)$-conforming subspace of the broken space which may not exist for general polytopal meshes. 
For the works in this direction, we refer readers to \cite{2004HoustonPerugiaSchotzau,2005HoustonPerugiaSchneebeli} for interior penalty DG methods and \cite{2019ChenCuiXu,2020ChenMonkZhang} for HDG methods.
Such a fact renders one of the major features of DG-type methods - their ability to be used on polytopal meshes - no longer applicable to the Maxwell equations.


\begin{figure}[h]
\centering
\includegraphics[width=5in]{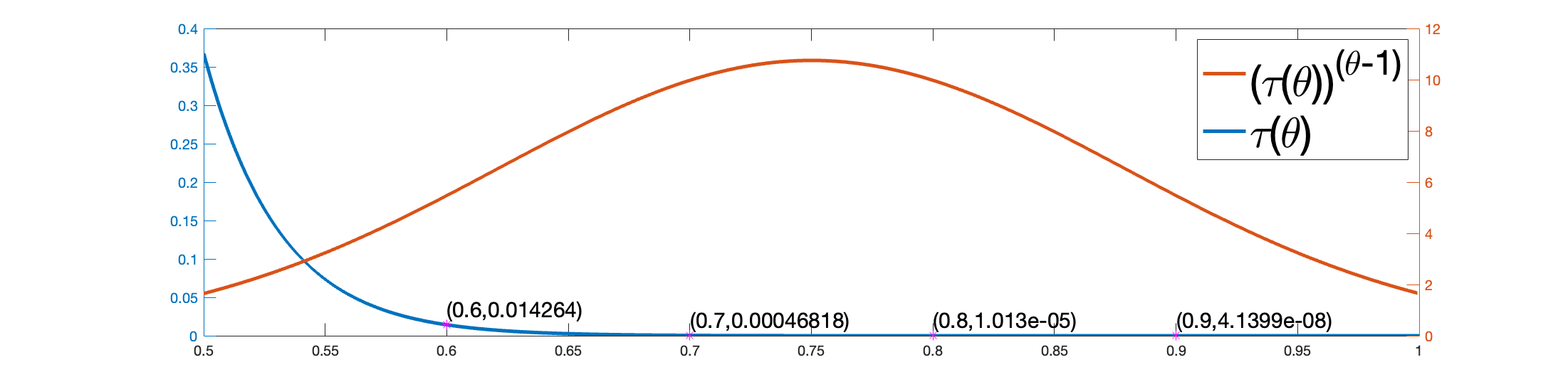}
 \caption{The shape regularity constant $\tau(\theta)$ and the error bound constant $(\tau(\theta))^{(\theta-1)}$ v.s. the solution regularity order $\theta$. 
 Some specific values of $\tau(\theta)$ are shown on the plots for $\theta = 0.6, 0.7, 0.8, 0.9$.
 It can be observed that $\tau(\theta)$ can be quite small for relative larger $\theta$, 
 thereby exhibiting its aptitude for accommodating practical applications.}
\label{fig:tauplot}
\end{figure}


\subsection{The main results of this paper}

The present work aims to demonstrate, through both theoretical analysis and numerical experiments, 
that the VEMs can achieve the optimal convergence rate for time-harmonic Maxwell equations not only on polygonal meshes, but also on highly anisotropic ones. 
There have been a few works on studying VEMs for electromagnetic problems in the literature, 
see \cite{2020BeiroMascotto,2016VeigaBrezziMarini,2020BeiroMascotto} for the construction and analysis of $\bfH(\rot)$-conforming virtual spaces, 
\cite{2017VeigaBrezziDassiMarini} for a magnetostatic model, and \cite{BEIRAODAVEIGA2021} for a time-dependent Maxwell equation.
However, to our best knowledge, due to the low regularity and potentially indefiniteness,
the application and analysis of VEMs to time-harmonic Maxwell equations still remain largely unexplored.

Here, we summarize the contributions of this work below:
\begin{itemize}

\item We present the first analysis of VEMs for the indefinite time-harmonic Maxwell equations with relatively low regularity $\bfH^{\theta}$, $\theta\in (1/2,1]$. 
We extend the classical theories of Helmholtz decomposition and {Hodge} mapping to the virtual spaces and an approximated $L^2$ inner product involving an optimally scaled stabilization. 
We show that the newly developed Hodge mapping can also lead to the optimal approximation accuracy even on highly anisotropic meshes.
We prove that, with the duality argument, these new techniques can lead to the optimal error estimates.

\item 
The central technique used in the present analysis to handle polygonal shapes is significantly different from those in the literature that look for a conforming subspace of the broken space which, 
heavily inevitably demands simplicial meshes. 
Instead, it is based on a ``virtual mesh" technique that yields a conforming virtual space on almost arbitrary polygonal meshes.
Furthermore, the geometric assumption for this work is delicate: 
on one hand, we construct a function $\tau(\theta)\in (0,0.5]$ depending on the regularity order $\theta$, and assume a polygonal element to satisfy
\begin{equation}
\label{assump_polygon0}
\rho_K \ge \tau(\theta) h_K
\end{equation}
where $h_K$ and $\rho_K$ are diameter of $K$ and the radius of the ball to which $K$ is convex, respectively. 
Here, $\tau(\theta)$ has the property that $\lim_{\theta\rightarrow 1/2}\tau(\theta) = \mathcal{O}(1)$ but $\lim_{\theta\rightarrow 1}\tau(\theta) = 0$, 
indicating that $K$ can be arbitrarily shrinking when $\theta \rightarrow 1$. 
On the other hand, for \eqref{assump_polygon0}, 
our novel techniques can show that the error bound merely involves a constant $(\tau(\theta))^{\theta-1}$ which is uniformly bounded for any $\theta\in(1/2,1]$. 
In contrast, the conventional analysis can only lead to $(\tau(\theta))^{-1}$ which blows up as $\tau(\theta) \rightarrow 0$. 
Particularly, we refer readers to Figure \ref{fig:tauplot} for illustration.

\item We conduct extensive numerical experiments, besed on the FEALPy
    package \cite{fealpy}, to demonstrate that the proposed algorithm is capable of effectively handling complex media interfaces that incorporate intricate topology, geometric singularities, and thin layers. 
These situations are typically considered challenging in practice.

\end{itemize}

It is worth noting that the anisotropic analysis in this work essentially recovers the error analysis based on the maximum angle condition \cite{1976BabuskaAziz,1999Duran,2020KobayashiTsuchiya} for the standard simplicial meshes, 
which allow extremely narrow and thin elements. 
Specifically, we refer interested readers to \cite{2005BuffaCostabelDauge,2002Nicaise,2011Lombardi} for the maximum angle condition of edge elements. 
Recently, the maximal angle condition has been also extended to polyhedral elements in \cite{2023Guo}.
In the instances of such irregular elements, the conventional scaling argument based on affine mappings is not even available for the simplicial meshes, let alone polygonal meshes. 
Furthermore, the aforementioned works all require higher regularity, say at least $\bfH^{1+\epsilon}$ for the edge elements. 
Our result substantially alleviates this constraint of regularity while simultaneously unveiling a profound connection between the element geometry and solution regularity.



In the next section, we introduce some basic geometric assumptions and Sobolev spaces. 
In Section \ref{sec:vspace}, we describe the virtual spaces and prepare some of their fundamental properties which will be frequently used in this work. 
In Section \ref{sec:erroreqn}, we prove the optimal error estimates. 
In Section \ref{sec:lemmaproof}, we show the analysis of two technical lemmas as well as the construction of $\tau(\theta)$ in \eqref{assump_polygon0}.
In the last section, we present extensive numerical experiments to demonstrate the effectiveness of the proposed method.



\section{Preliminaries}

Let $\Omega$ be a simply connected domain. 
In this work, we do not assume any smoothness property for the interface. 
Instead, we make the following assumptions:
\begin{assumption}
\label{assum_geo}
The interface satisfies the following conditions
\begin{itemize}
\item[(I1)]\label{asp:I1} it does not intersect itself;
\item[(I2)]\label{asp:I2} the regular (plump) property: there exists a constant $c_0$ such that
\begin{equation}
\label{Dregular}
|B(\bfx,r)\cap \Omega^{\pm}| \ge c_0 r^2, ~~~ \forall \bfx \in \Omega^{\pm}, ~~ r\in[0,1],
\end{equation}
where $B(\bfx,r)$ is a ball centering at $\bfx$ with the radius $r$.
\end{itemize}
\end{assumption}
Note that \hyperref[asp:I1]{(I1)} in Assumption \ref{assum_geo} basically means that a corner of $\Gamma$ and $\partial\Omega$ must be lower bounded and upper bounded from $0$ and $\pi$, respectively.

Let $\mathcal{T}_h$ be a polygonal mesh conforming or fitted to the interface and $\partial\Omega$ in the sense that each element resides on only one side of the interface. 
Nevertheless, thanks to the considerable flexibility in polygonal element shape permitted in this work, 
$\mathcal{T}_h$ can be generated in a ``non-conforming" manner, for example,
(i) semi-structure meshes that are cut by arbitrary interface from background Cartesian meshes \cite{2021CaoChenGuo,2017ChenWeiWen}, 
see the left two plots in Figure \ref{fig:NonConfMesh};  
(ii) non-conforming meshes utilized for distinct objects within a domain, see the right two plots in Figure \ref{fig:NonConfMesh}.

\begin{figure}[h]
\centering
\begin{subfigure}{.2\textwidth}
    \includegraphics[width=1.05in]{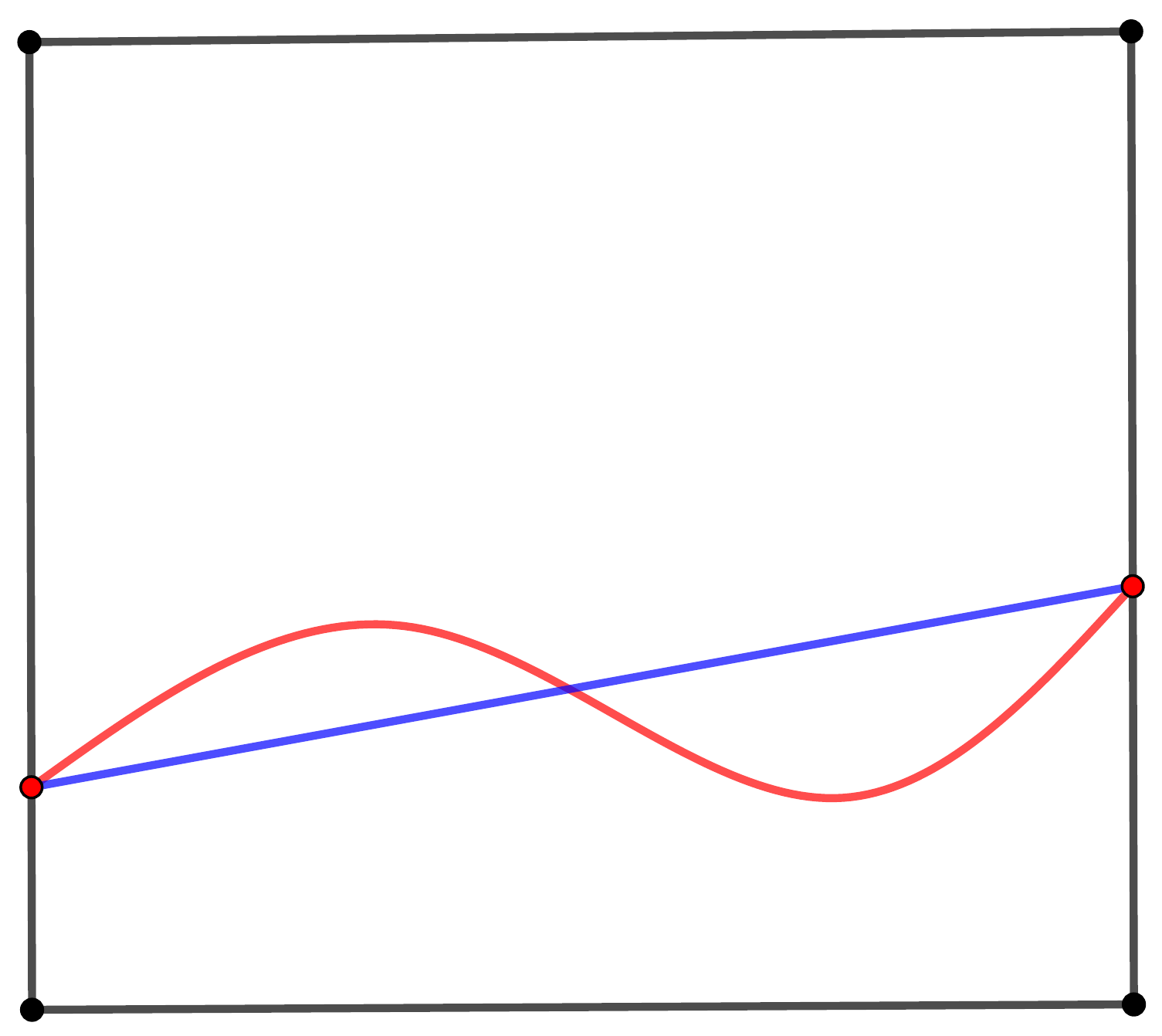}
     \label{fig:interfelem1} 
\end{subfigure}
 \begin{subfigure}{.2\textwidth}
    \includegraphics[width=1.05in]{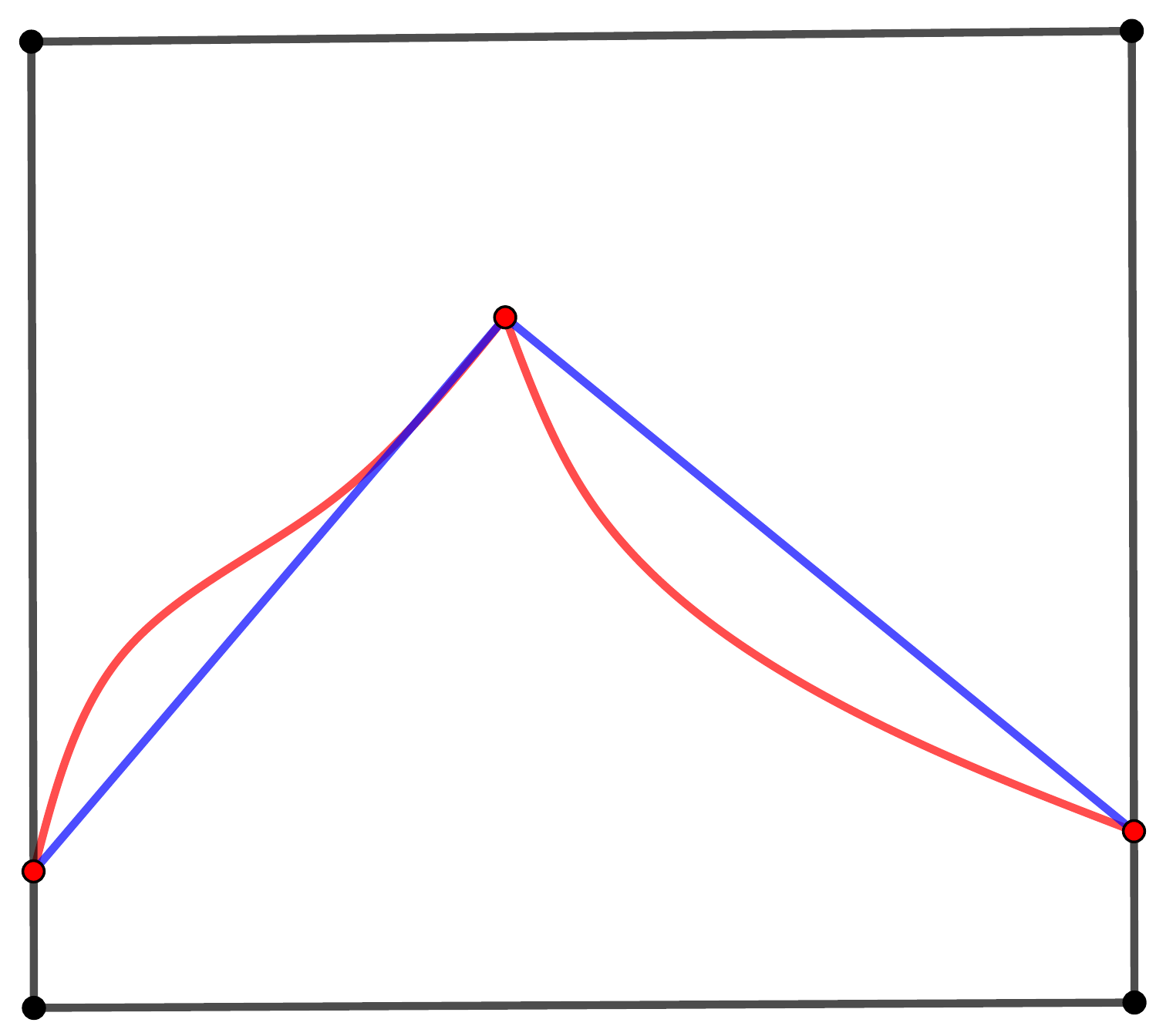}
     \label{fig:interfelem2} 
\end{subfigure}
\begin{subfigure}{.2\textwidth}
    \includegraphics[width=1.05in]{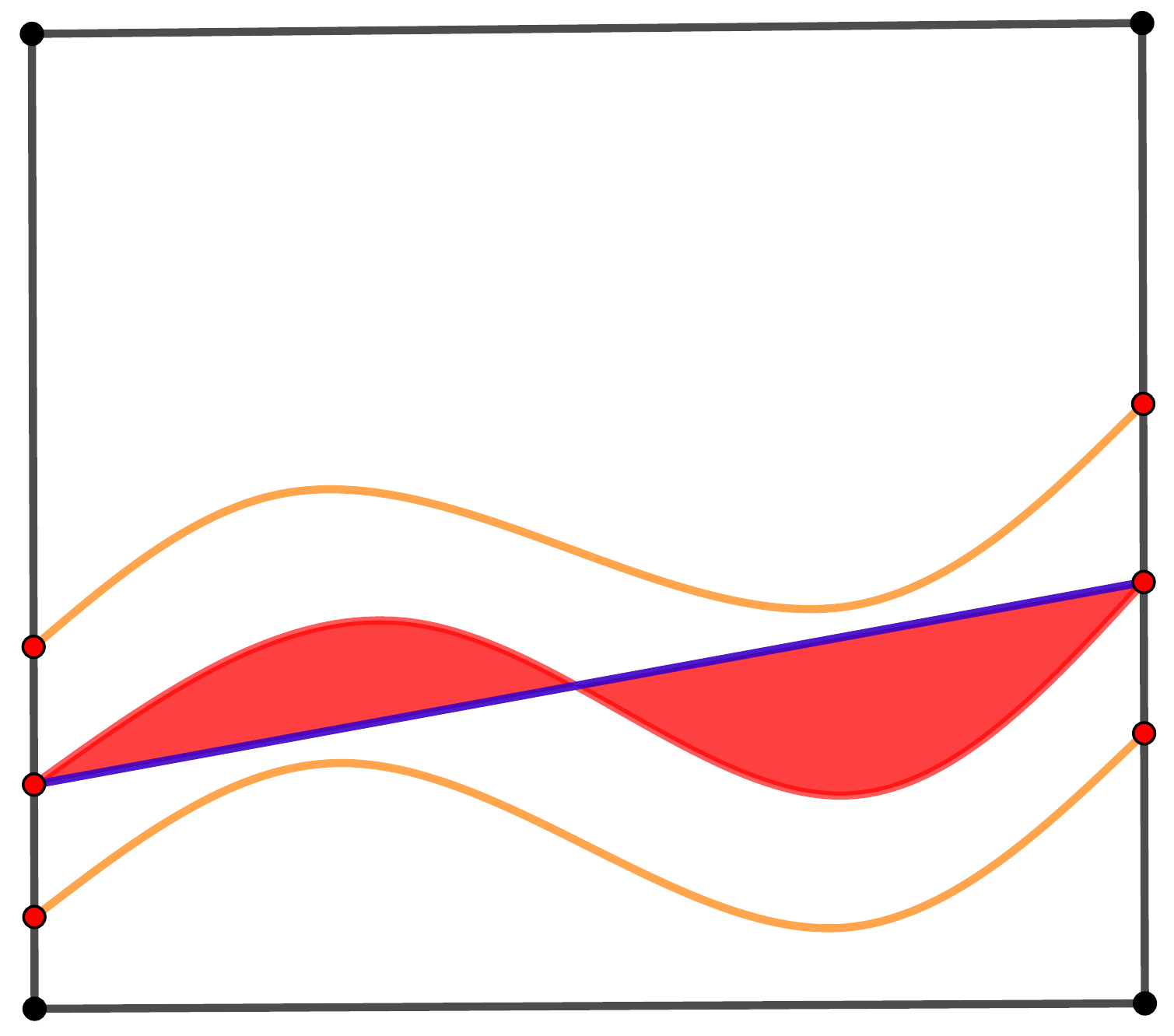}
     \label{fig:interfelem3} 
\end{subfigure}
     \caption{Left: the interface (red curve) cuts an element $K$ into two subelements; construct a segment as $\Gamma^K_h$ (blue segment) by connecting the intersection points. 
     Middle: use a polyline $\Gamma^K_h$ to fit a piecewise smooth interface curve. 
     Right: the mismatched portion by $\Omega^{\pm}$ and $\Omega^{\pm}_h$ which is contained in a thin strip $\Omega_{\epsilon h^2}$.
      }
  \label{fig:interfelem} 
\end{figure}

Let $\Gamma_h$ be the polygonal approximation to the exact interface $\Gamma$ that participates $\Omega$ into $\Omega^{\pm}_h$. 
Let $\alpha_h$ and $\beta_h$ be defined with $\Gamma_h$. 
Note that $\Gamma_h$ is formed by edges of elements in the meshes, see Figure \ref{fig:interfelem} for illustration. 
There is a mismatching region between $\Gamma_h$ and $\Gamma$ formally defined as $(\Omega^+_h\cap \Omega^-)\cup (\Omega^-_h\cap \Omega^+)$ which should be carefully treated in the analysis. 
For this purpose, we introduce a region that is near the interface.
\begin{equation}
\label{OmegaGamma}
\Omega^{\Gamma}_{\epsilon} = \{ \bfx\in \Omega: \dist(\bfx,\Gamma)< \epsilon \},
\end{equation}
and make the following assumption regarding the mismatching portion.
\begin{assumption}
\label{interf_eps}
There exists $\epsilon>0$ independent of the mesh only depending on $\Gamma$ such that
\begin{equation}
\label{interf_eps_eq}
(\Omega^+_h\cap \Omega^-)\cup (\Omega^-_h\cap \Omega^+) \subseteq \Omega^{\Gamma}_{\epsilon h^2}.
\end{equation}
\end{assumption}
See the right plot in Figure \ref{fig:interfelem} for the illustration of this assumption.

Let $H^s(D)$, $\bfH^s(D)$, $\bfH(\rot;D)$ and $\bfH(\ddiv;D)$, $s\ge 0$, be the standard (vector) Sobolev spaces on a region $D$.
We will employ the following fractional order Sobolev norm for $s\in(0,1)$:
\begin{equation}
\label{fracSobolev}
 |v|_{s,D} = \left(\int_D \int_D \frac{ |u(\bfx) - u(\bfy)|^2 }{ |\bfx -
\bfy|^{2(1+s)} } \dd \bfx \dd \bfy\right)^{\frac{1}{2}}, ~~~~ \text{and} ~~~~ \| v \|_{s,D} = \| v \|_{0,D} + |v|_{s,D}.
\end{equation}
In addition, for simplicity of presentation, we shall introduce the following spaces involving possibly variable and discontinuous coefficients:
\begin{subequations}
\label{spaces}
\begin{align}
    & \bfH(\dd^0;\Omega) = \{ \bfu \in \bfH(\dd;\Omega) : \dd \bfu = 0 \} ~~~ \dd = \rot ~ \text{or} ~ \ddiv,  \label{spaces1} \\
    & \bfH(c, \dd  ;\Omega) = \{ \bfu: c \bfu \in \bfH(\dd;\Omega)  \} ~~~ \dd = \rot ~ \text{or} ~ \ddiv \label{spaces2} ,
\end{align}
\end{subequations}
where $c$ is a (piecewise) constant function. The subscript ``0" of these spaces then denotes subspaces with the zero traces. 
Furthermore, the traces involved in \eqref{spaces} on $\partial \Omega$ and $\Gamma$ are all understood in the $L^2$ sense in order for deriving meaningful regularity estimates \cite{1999MartinMoniqueSerge}. 
In addition, we shall frequently use the following $\beta$-div-free subspace of $\bfH_0(\rot;\Omega)$:
\begin{equation}
\label{div_free}
\bfK(\beta;\Omega) = \bfH_0(\rot;\Omega) \cap \bfH(\beta,\ddiv^0;\Omega) \}.
\end{equation}
So $\bfK(\beta;\Omega)$ is $\beta$-weighted $L^2$ orthogonal to the space $\nabla H^1_0(\Omega)$, i.e., in the sense of $(\beta\cdot,\cdot)_{\Omega}$.
Furthermore, we define
\begin{equation}
\label{Yspace}
\bfY^{s}(\Omega) = \bfH^{s}(\Omega^{\pm})\cap \bfH(\beta,\ddiv;\Omega) \cap \bfH(\rot;\Omega).
\end{equation}
In fact, by some appropriate assumptions and choices of $s$, $\bfY^{s}(\Omega)$ can be simply characterized as $\bfH(\beta,\ddiv;\Omega) \cap \bfH(\rot;\Omega)$, see Lemma \ref{lem_embd_1}.


\section{The virtual scheme}
\label{sec:vspace}

In this section, we recall the $H^1$ and $\bfH(\rot)$ virtual spaces and prepare some fundamental estimates which will be frequently used.
Given each domain $D$, we let $\mathbb{P}_k(D)$ be the $k$-th degree polynomial space with $k$ being an integer. 

 \subsection{The virtual element method}
In \cite{2013BeiraodeVeigaBrezziCangiani,2016VeigaBrezziMarini}, the original $H^1$ and $\bfH(\rot)$-conforming virtual spaces are defined as
\begin{subequations}
\label{ive_space}
\begin{align}
  &   \widetilde{V}^n_h(K) = \{ v_h \in H^1(K):  \nabla v_h\in \bfH(\ddiv^0; K) \}, \label{ive_space1} \\
  &   \widetilde{\bfV}^e_h(K) = \{ \bfv_h \in \bfH(\rot;K) \cap \bfH(\ddiv^0;K):  \rot \bfv_h \in \mathbb{P}_0(K) \}. \label{ive_space2}
\end{align}
\end{subequations}
Note that the div-free property is only local. Here, we also refer interested readers to div-free virtual spaces in \cite{2021WeiHuangLi}.

As the local PDE and its stability are intimately intertwined with the element geometry, 
such definition is not suitable for anisotropic analysis. 
Instead, we introduce an equivalent definition that significantly aids in anisotropic analysis.
This novel definition draws extensively from the concept and technique of ``virtual meshes" originally introduced in \cite{2021CaoChenGuo}. 
We recall its definition and the associated assumptions below. 
For each element $K$, we let $\mathcal{N}_K$ and $\mathcal{E}_K$ be the sets of nodes and edges of $K$, respectively.
\begin{assumption}
\label{assump_vmesh}
There exists a local ``virtual mesh" $\mathcal{T}_h(K)$ satisfying
 \begin{itemize}
 \item[(L1)] \label{asp:polygonL1} A maximum angle condition: every angle of all triangles in $\mathcal T_h(K)$ is bounded by $\theta_M< \pi$.
 \item[(L2)] \label{asp:polygonL2} A short-interior-edge condition: every interior edge $e$ either satisfies that $h_e \ge d_1 h_K$ or that there exist $e_l\in\mathcal{E}_K$, $l=1,2,...,L$, connecting the two ending points of $e$ such that $h_e\ge d_2 h_{e_l}$.
 \item[(L3)] \label{asp:polygonL3} A no-interior-node condition: the collection of edges $\mathcal{E}_h(K)$ in $\mathcal{T}_h(K)$ are solely formed by the vertices on $\partial K$.
\end{itemize}
\end{assumption}
\begin{remark}
\label{rem_auxiliary_space}
We make the following remarks for the assumptions above.
\begin{itemize}

\item The local mesh here is referred to as a ``virtual mesh" because it is only used for analysis purposes and is not needed for computation. 
In other words, there is no need to actually form the triangulation when implementing the proposed method.

\item A triangulation satisfying \hyperref[asp:polygonL1]{(L1)}-\hyperref[asp:polygonL3]{(L3)} usually exists for a large class of element shapes, 
including extremely thin and narrow elements. For example, those cases in  Figure \ref{fig:submesh}.
It is worth noting that these shrinking elements pose a significant challenge for many methods in the literature.
\end{itemize}
\end{remark}

\begin{figure}[h]
\centering
\begin{subfigure}{.2\textwidth}
    \includegraphics[width=0.5in]{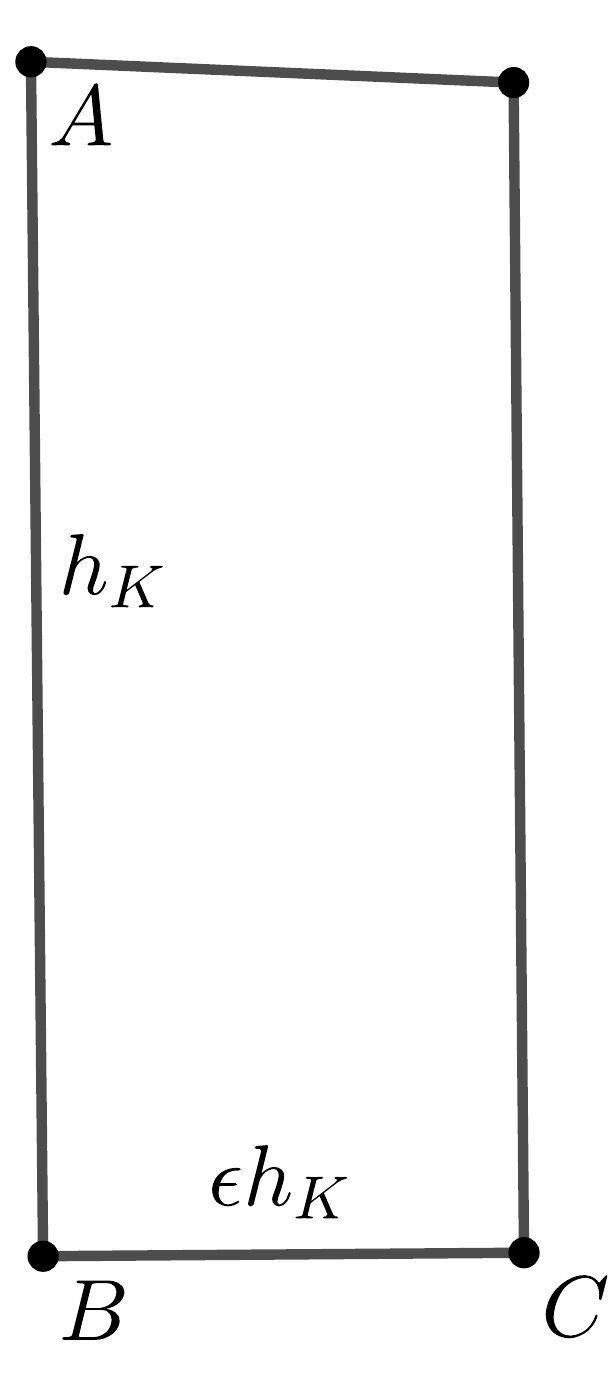}
     \label{fig:submesh5} 
\end{subfigure}
\begin{subfigure}{.2\textwidth}
    \includegraphics[width=1.05in]{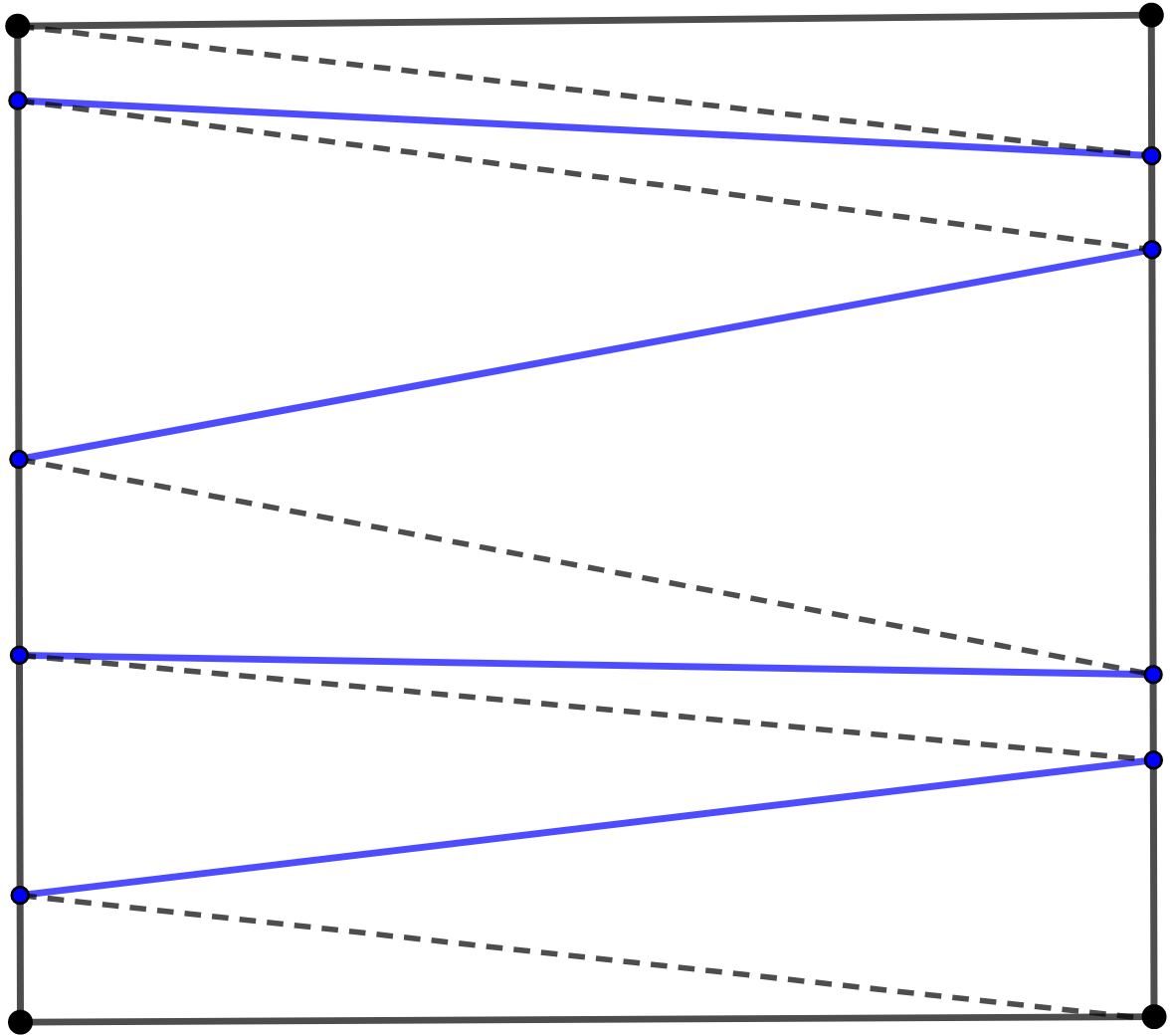}
     \label{fig:submesh8} 
\end{subfigure}
\begin{subfigure}{.2\textwidth}
    \includegraphics[width=1.05in]{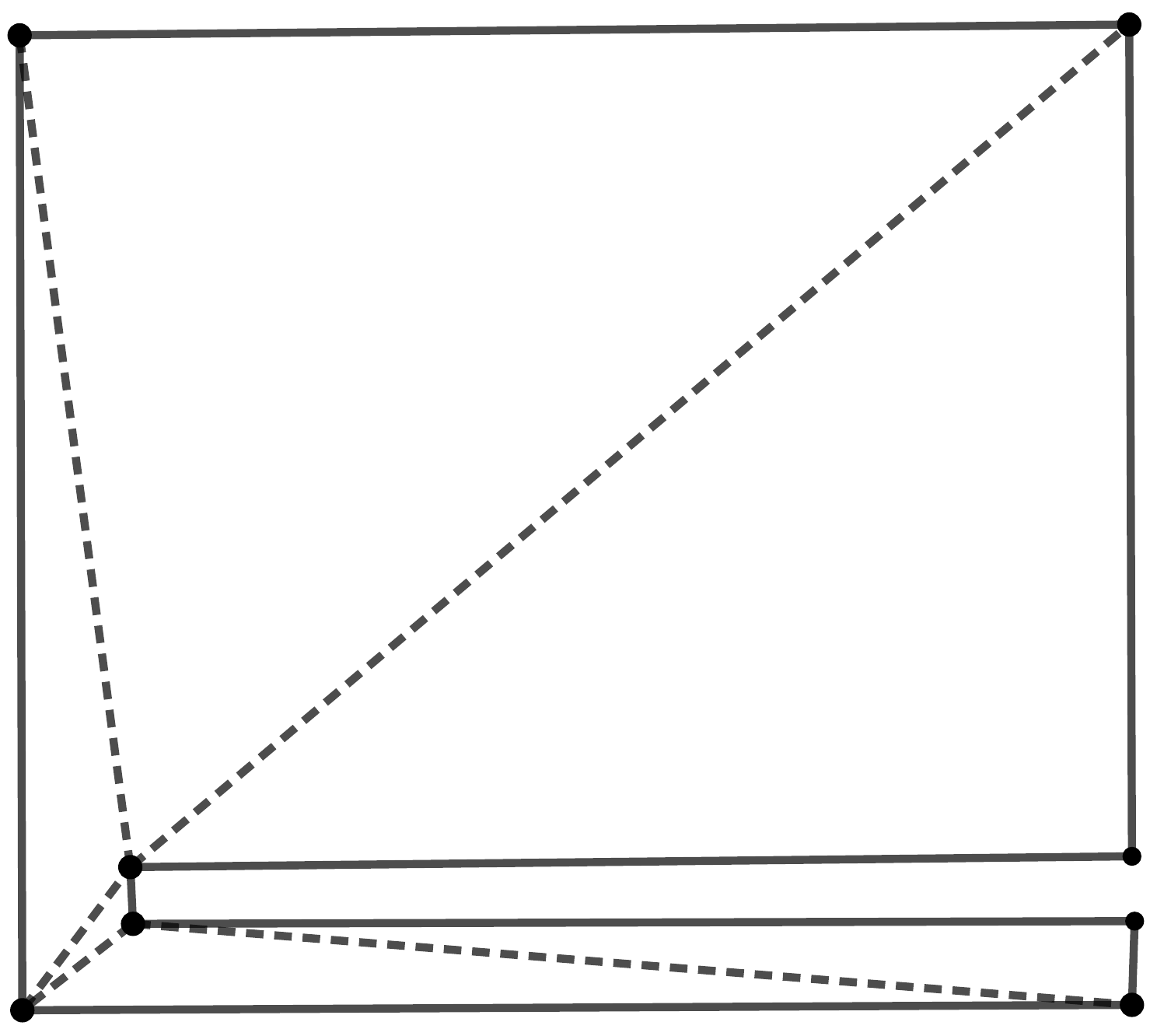}
     \label{fig:submesh4} 
\end{subfigure}
 \begin{subfigure}{.2\textwidth}
    \includegraphics[width=1.05in]{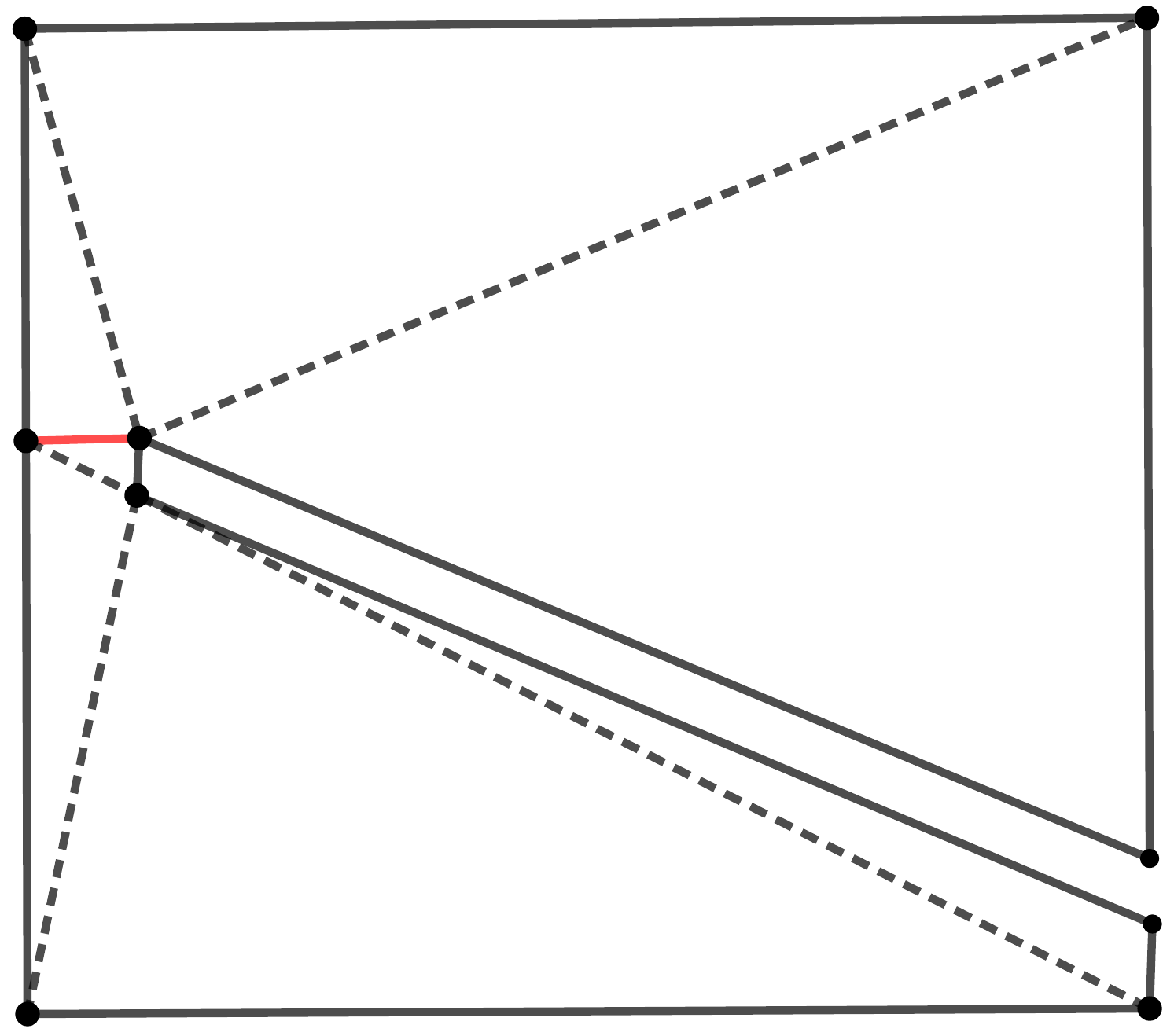}
     \label{fig:submesh3} 
\end{subfigure}
     \caption{The first plot: a thin element where the edge $AB$ has a height $h_K$ and the edge $BC$ has a length $\epsilon h_K$, with $\epsilon\rightarrow 0$; 
     and it satisfies Assumption \hyperref[assump_vmesh]{3} about the local virtual meshes as well as Assumption \hyperref[asp:polygonG2]{(G2)}. 
     Such a shape widely appears when a shape-regular element, say a square for example, is cut by interface (blue segments), as shown in the second plot. 
     The third plot: a square is cut by a thin layer that is very close to the left and bottom sides of the square; 
     and in this case, the local virtual mesh also satisfies Assumption \hyperref[asp:polygonG2]{(G2)}. 
     The fourth plot: a square is cut by a thin layer, and we need to artificially add one edge (the red edge) to separate the element into two such that the virtual mesh also satisfies \hyperref[asp:polygonL1]{(L1)} and \hyperref[asp:polygonL2]{(L2)}.}
  \label{fig:submesh} 
\end{figure}

Provided with the virtual mesh $\mathcal{T}_h(K)$, we define
\begin{subequations}
\label{auxi_v_space_1_polygon}
\begin{align}
V^n_h(K) & = \{ v_h \in H^1(K) : ~ v_h|_T\in\mathbb{P}_1(T), ~ \forall T \in \mathcal{T}_h(K) \},  \label{auxi_v_space_1_polygon1} \\
V^e_h(K) & = \{ \bfv_h \in \bfH(\curl ;K): ~  \bfv_h|_{T} \in \mathcal{ND}_h(T), ~ \forall T\in \mathcal T_h(K), ~  \curl \bfv_h \in \mathbb{P}_0(K)\}.\label{auxi_v_space_1_polygon2}
\end{align}
\end{subequations}
In fact, the nodal space $V^n_h(K)$ is the just the classical Lagrange element spaces defined on the virtual mesh. 
It can be immediately shown that the spaces described above have the same nodal and edge degrees of freedom (DoFs) as the original spaces in \eqref{ive_space}, 
ensuring isomorphism between these two distinct families of spaces. 
As a result, the projection operators and the resulting computation scheme remain completely unchanged.
Furthermore, the edge space in \eqref{auxi_v_space_1_polygon2} can be regarded as the solution space to the following discrete local problem:
for any given boundary conditions $\bfv_h\cdot\bft_e$, $e\in \mathcal{E}_K$, 
find $(\bfv_h,\lambda_h)\in \mathcal{ND}_{h,0}(K)\times V^n_{h,0}(K)$ satisfying
\begin{equation}
\label{harmon_ext_eq_1} 
\left \{
\begin{aligned}
    &  (\rot \bfv_h, \rot \bfw_h )_K + (\bfw_h ,\nabla \lambda_h)_K= 0, ~~~~ \forall \bfw_h \in \mathcal{ND}_{h,0}(K), 
 \\
    &  (\bfv_h,\nabla p_h)_K =0, ~~~~ \forall p_h \in V^n_{h,0}(K),
\end{aligned}
\right.
\end{equation}
where $\mathcal{ND}_{h,0}(\mathcal{T}_h(K))$ is the N\'ed\'elec element space with the zero trace, 
and $V^n_{h,0}(K)$ is the subspace of $V^n_{h}(K)$ also with the zero trace which is trivial as no interior node exists in the virtual mesh. 
We refer interested readers to \cite[Lemma 3.2]{2021CaoChenGuo} for a detailed discussion. 

Thanks to the nodal and edge DoFs, the global spaces can be defined as
\begin{subequations}
\label{glob_space}
\begin{align}
  &   V^n_h = \{ v_h|_K \in V^n_h(K),~ \forall K\in \mathcal{T}_h \} \cap \bfH_0(\rot;\Omega) ,  \\
  &   \bfV^e_h = \{ \bfv_h|_K \in \bfV^e_h(K),~ \forall K\in \mathcal{T}_h \} \cap H^1_0(\Omega) .
\end{align}
\end{subequations}
We also introduce the piecewise constant space $Q_h = \{v_h\in L^2(\Omega): v_h|_K \in \mathbb{P}_0(K) \}$. 
Furthermore, we can define the following interpolations
\begin{subequations}
\label{interp}
\begin{align}
  & I^n_h :  H^2(\beta;\mathcal T_h) \longrightarrow V^n_h, ~~~ I^n_hv_h(\bfx) = v_h(\bfx), ~ \forall \bfx\in \mathcal{N}_h,  \\
  & I^e_h : \bfH^1(\rot,\alpha,\beta;\mathcal T_h) \longrightarrow \bfV^e_h, ~~~ \int_e I^n_h \bfv_h\cdot\bft \dd s = \int_e \bfv_h\cdot\bft \dd s, ~ \forall e\in \mathcal{E}_h.
\end{align}
\end{subequations}
The $L^2$ interpolation can be defined as $\Pi^k_D: L^2(D)\rightarrow \mathbb{P}_k(D)$ will be also frequently used. 
As only the lowest order case is considered in this work, we shall focus on $\Pi^0_D$ and drop ``$0$" for simplicity. 
Define the global interpolation $\Pi_h|_K = \Pi_K$.

Provided with the spaces and operators above, we form the following diagram:
\begin{equation}
\label{DR_curl}
\left.
\begin{array}{ccccccc}
\mathbb R \xrightarrow[]{\quad} &     H^2(\Omega^{\pm})\cap H^1_0(\Omega) & \xrightarrow[]{~~\nabla~~} & \bfH^1(\Omega^{\pm}) \cap \bfH_0(\rot;\Omega) & \xrightarrow[]{~~\rot~~} & L^2(\Omega)  &\xrightarrow[]{\quad} 0  \\
~~~~& \quad \bigg\downarrow I^n_{h}  &  & ~~~~\bigg\downarrow I^e_{h}  &  &~~~~\bigg\downarrow \Pi_h \\
\mathbb R \xrightarrow[]{\quad}& V^n_{h} & \xrightarrow[]{~~\nabla~~} & \bfV^e_{h} & \xrightarrow[]{~~\rot~~} & Q_h  &\xrightarrow[]{\quad} 0.
\end{array}\right.
\end{equation}
\begin{lemma}
\label{lem_exact_seq}
The diagram \eqref{DR_curl} is commutative,
and the sequence on the bottom of \eqref{DR_curl} is exact.
\end{lemma}
\begin{proof}
As the spaces here are isomorphic to the classic ones in \eqref{ive_space} through their DoFs, the argument is just identical to \cite{2016VeigaBrezziMarini}.
\end{proof}

Similar to the ``virtual" spirit of virtual spaces, the virtual mesh is treated as ``non-computable".
In computation, following the original VEM, we need to project them to constant vector spaces. 
Indeed, it can be shown that a weighted $L^2$ projection
$\Pi_K:(L^2(K))^2 \rightarrow (\mathbb{P}_0(K))^2$ defined as
\begin{equation}
\label{ife_projection_curl_1}
( \Pi_K \bfu, \bfv_h)_K  = (  \bfu, \bfv_h)_K, ~~~ \forall \bfv_h\in (\mathbb{P}_0(K))^2,
\end{equation}
is computable.
Here, even though the projection here is for vector spaces, we shall keep the same notation $\Pi_K$ for simplicity. 
In addition, $\rot \bfv_h$, $\bfv_h\in \bfV^e_h(K)$, can be explicitly computed through integration by parts and DoFs. 


Then, the VEM is to find $\bfu_h\in \bfV^e_h$ such that
  \begin{equation}
\label{vem0}
a_h(\bfu_h,\bfv_h) = (\bff,\Pi_h \bfv_h)_{\Omega}, ~~~~~ \forall \bfv_h\in \bfV^e_h,
\end{equation}
where
\begin{subequations}
\label{vem1}
\begin{align}
    &   a_h(\bfu_h,\bfv_h) := (\alpha_h \rot\bfu_h, \rot\bfv_h)_{\Omega} - b_h(\bfu_h,\bfv_h),  \label{vem11}   \\
    &   b_h(\bfu_h,\bfv_h) := \sum_{K\in\mathcal{T}_h} b_K(\bfu_h,\bfv_h),  \label{vem12}  \\
    &    b_K(\bfu_h,\bfv_h) := (\beta_h\Pi_K \bfu_h,\Pi_K \bfu_h )_{K} + h ( (I-\Pi_K )\bfu_h\cdot\bft ,  (I-\Pi_K )\bfu_h\cdot\bft )_{\partial K}.  \label{vem13}
\end{align}
\end{subequations}
At last, we shall introduce the following norms for $ \bfv\in \bfH^{s}(\rot;\Omega^{\pm})\oplus \bfV^e_h, s>1/2$:
\begin{equation}
\label{abnorm}
\| \bfv \|^2_b:= b_h(\bfv,\bfv) ~~~~ \text{and} ~~~~\| \bfv \|^2_a := \| \rot\bfv \|^2_{0,\Omega} + \|\bfv \|^2_b.
\end{equation}

\subsection{Basic estimates}

In this subsection, we present several estimates which will be frequently used in the subsequent analysis. 
We first recall the thin-strip argument \cite[Lemma 2.1]{2010LiMelenkWohlmuthZou} which enables us to handle the problematic region $\Omega^{\Gamma}_{\epsilon h^2}$ in Assumption \ref{interf_eps}:
\begin{lemma}
\label{lem_strip}
For $z\in H^{s}(\Omega^-\cup\Omega^+)$, $s>1/2$, there holds
\begin{equation}
\label{lem_strip_eq0}
\| z \|_{0,\Omega^{\Gamma}_{\epsilon h^2}} \lesssim \sqrt{\epsilon} h \| z \|_{s,\Omega^{\pm}}.
\end{equation}
\end{lemma}


Next, we present the following results regarding the positivity of $b_h(\cdot,\cdot)$ and $a_h(\cdot,\cdot)$. 
\begin{lemma}
\label{bn_norm}
$b_h(\cdot,\cdot)$ defines an inner product on $\bfV^e_h$, and consequently $\|\cdot\|_b$ is a norm.
\end{lemma}
\begin{proof}
We only need to verify that $\|\bfv_h\|_b=0$ implies $\bfv_h=\mathbf{ 0}$ for $\bfv_h \in \bfV^e_h$. 
In fact, we immediately have $\Pi_K\bfv_h = \mathbf{ 0}$ on each $K$, and thus $\bfv_h\cdot\bft = 0$ on $\partial K$. 
Due to the DoFs, we have $\bfv_h=\mathbf{ 0}$.
\end{proof}

The following two Lemmas are critical for our analysis; 
but due to technicalities arising from low regularity and anisotropic element shape, 
we postpone their proof to Sections \ref{sec:lemmaproof}.
\begin{lemma}
\label{lem_Pi_est}
Let $\bfu\in \bfY^s(\Omega)$, $s\in (1/2,1]$.  
Under Assumptions \ref{interf_eps}, \ref{assump_vmesh} and \ref{assump_polygon}, there holds that
\begin{equation}
\label{lem_Pi_est_eq0}
 \sum_{K\in\mathcal{T}_h} \| \bfu - \Pi_K \bfu \|^2_{0,K} + h_K\| (I - \Pi_K )\bfu\cdot \bft \|^2_{0,\partial K} \lesssim h^{2s} \| \bfu \|^2_{s,\Omega^{\pm}}.
\end{equation}
\end{lemma}
\begin{lemma}
\label{lem_interp_est}
Let $\bfu\in \bfY^s(\Omega)$, $s\in (1/2,1]$. 
Under Assumptions \ref{interf_eps}, \ref{assump_vmesh} and \ref{assump_polygon}, there hold that
\begin{subequations}
\label{lem_interp_est_eq0}
\begin{align}
    & \| \bfu - I^e_h \bfu \|_{0,\Omega}  \lesssim h^{s}  \| \bfu \|_{s,\Omega^{\pm}} + h\|  \rot \bfu \|_{0,\Omega} ,  \label{lem_interp_est_eq01}\\
    & \| \bfu - I^e_h \bfu \|_{b} \lesssim h^{s}  \| \bfu \|_{s,\Omega^{\pm}} + h\| \rot \bfu \|_{0,\Omega}.   \label{lem_interp_est_eq02}
\end{align}
If there additionally holds $\alpha\rot\bfu\in \bfH^{s}(\Omega)$, then
\begin{align}
  & \| \rot( \bfu - I^e_h \bfu ) \|_{0,\Omega} \lesssim h^{s}  \| \rot \bfu \|_{s,\Omega^{\pm}} .   \label{lem_interp_est_eq03}
\end{align}
\end{subequations}
\end{lemma}


One critical component in the analysis of VEMs is the boundedness of the mesh-dependent norms, 
which can be described as the continuity of the identity operators between the related spaces.

\begin{lemma}
\label{lem_L2stab}
Under Assumption \ref{assump_vmesh}, the following identity operators are continuous
\begin{subequations}
\label{lem_L2stab_eq0}
\begin{align}
    & (\bfV^e_h, \| \cdot \|_{a}) \xrightarrow[]{~~I~~}  (\bfV^e_h,\|\cdot\|_{0,\Omega}), \label{lem_L2stab_eq01} \\
    &  ( \bfY^{s}(\Omega) ,\|\cdot\|_{s,\Omega^{\pm}} ) \xrightarrow[]{~~I~~}  (\bfY^s(\Omega),\|\cdot\|_{b}), ~~~ s\in(1/2,1]. \label{lem_L2stab_eq02}
\end{align}
\end{subequations}
\end{lemma}
\begin{proof}
For \eqref{lem_L2stab_eq01}, following the argument of \cite[Lemma 3.4]{2021CaoChenGuo}, we can show
\begin{equation}
\label{lem_Poincare_eq0}
\| \bfv_h \|_{0,K} \le  \left( \frac{ \cot(\theta_{M}) C_K  }{2\min\{d_1,d_2\}} \right)^{1/2} (  h^{1/2}_{K} \|  \bfv_h\cdot\bft \|_{0,\partial P} + h_{K} \| \rot\, \bfv_h \|_{0,K}),
\end{equation}
where  $C_K$ is an integer only depending on the number of vertices of $K$, and $d_1$ and $d_2$ come from  \hyperref[asp:polygonL2]{(L2)} in Assumption \ref{assump_vmesh}. 
Applying \eqref{lem_Poincare_eq0} to $\bfv_h - \Pi_K \bfv_h$ and using the boundedness of $\Pi_K$ yields \eqref{lem_L2stab_eq01}.

Next, for \eqref{lem_L2stab_eq02}, given $\bfv\in \bfY^s(\Omega)$, using Lemma \ref{lem_Pi_est} and the boundedness property of the projection, we have
\begin{equation}
\label{lem_L2stab_eq2}
\| \bfv \|^2_b = \| \Pi_h \bfv \|^2_{0,\Omega}  +  \sum_{K\in\mathcal{T}_h} h_K \| (I-\Pi_K )\bfv\cdot\bft \|^2_{0,\partial K} \lesssim \| \bfv \|^2_{0,\Omega} + h^{2s} \| \bfv \|^2_{s,\Omega^{\pm}} .
\end{equation}
\end{proof}


\begin{remark}
It should be noted that while $b_h(\cdot,\cdot)$ serves as an approximation to the $L^2$ inner product, 
Lemma \ref{bn_norm} solely ensures its positive definiteness without providing a uniformly lower bound that is uniformly robust to variations in element shape. 
To achieve such a desirable robust bound, the inclusion of the $\rot$ term is essential, as demonstrated in \eqref{lem_L2stab_eq01}.
In addition, bounding $\|\cdot\|_b$ by the $L^2$ norm lacks robustness to the element shape. 
Instead, we need to enhance regularity to bound it by the $\|\cdot\|_{s,\Omega^{\pm}}$, as illustrated by \eqref{lem_L2stab_eq02}.
\end{remark}



\subsection{Decomposition and embedding}
\label{sec:decomp}
In this subsection, we broaden the scope of some well-established results of regular decomposition to accommodate discontinuous coefficients and the inner products associated with the VEMs introduced above.
For the sake of completeness, we start with the following $H^1$-type elliptic equation:
\begin{equation}
\label{H1eqn}
\ddiv(\beta\nabla u) = f ~~ \text{in} ~~ \Omega, ~~~~ [u]_{\Gamma} =0, ~~~~ [\beta\nabla u\cdot\bfn]_{\Gamma}=g, ~~ \text{on} ~~ \Gamma, ~~~~ u=0 ~~ \text{on} ~~ \partial\Omega,
\end{equation}
with $f\in L^2(\Omega)$ and $g\in H^{1/2}(\Omega)$. 
Under condition \hyperref[asp:I1]{(I1)} in Assumption \ref{assum_geo}, Theorem 16 in \cite{2016Ciarlet} indicates that $u\in H^{1+\theta}_0(\Omega^{\pm})$ with $\theta\in(1/2,1]$.
If the interface intersects itself, the regularity may be even lower, i.e., $\theta\in(0,1)$, see the discussions in \cite{1974Kellogg,Nicaise1994,2001Petzoldt,2016Ciarlet}.
Hence, the following embedding result holds.

\begin{lemma}
\label{lem_embd_1}
Under \hyperref[asp:I1]{(I1)} in Assumption \ref{assum_geo}, there exists $\theta\in(1/2,1]$ such that 
\begin{equation}
\label{lem_embd_1_eq0}
\bfH_0(\rot;\Omega)\cap \bfH(\ddiv,\beta;\Omega) \hookrightarrow \bfH^{\theta}(\Omega^{\pm}).
\end{equation}
\end{lemma}
\begin{proof}
With the elliptic regularity for the elliptic equation \eqref{H1eqn}, 
the result for the real $\beta$ can be found in \cite[Theorem 7]{2016Ciarlet}. 
Here for the sake of completeness, we also show the case of complex $\beta$ in Appendix \ref{append_embed}, 
despite only the real case being used in this work.
\end{proof}

With Lemma \ref{lem_embd_1}, we immediately have 
\begin{equation}
\label{lem_embd_2}
\bfK(\beta;\Omega)  \hookrightarrow \bfH^{\theta}(\Omega^{\pm}),
\end{equation}
with $\theta$ specified in Lemma \ref{lem_embd_1}. Next, we introduce the \textit{Hodge} mapping
\begin{equation}
\label{hodge_map}
\mathcal{H}_{\beta}: \bfV^e_h \rightarrow \bfK(\beta;\Omega) ,~~~~\text{satisfying} ~~ \rot \mathcal{H}_{\beta}\bfw_h = \rot \bfw_h.
\end{equation}
Note that \eqref{hodge_map} is well-defined as $\beta$ is positive. By \eqref{lem_embd_2}, there holds $\mathcal{H}_{\beta}\bfV^e_h \subset \bfH^{\theta}(\Omega^{\pm})$. 

Similar to the classic N\'ed\'elec spaces, the $\bfH(\rot)$-conforming virtual spaces do not have the global $\ddiv$-free property.
We need to consider a discrete-div-free condition in the sense of the inner product $b_h(\cdot,\cdot)$ (not the conventional standard $L^2$ inner product $(\cdot,\cdot)_{\Omega}$). 
It results in a special discrete Helmholtz decomposition for an arbitrary $\bfv_h\in \bfV^e_h$:
\begin{equation}
\label{Hdeomp1}
\begin{split}
&\bfv_h = \bfw_h + \nabla q_h, ~~~~ \bfw_h\in \bfV^e_h, ~~ q_h \in V^n_h, \\
\text{satisfying} ~~~& b_h( \bfw_h , \nabla s_h) = 0, ~~ \forall s_h \in V^n_h.
\end{split}
\end{equation}
It follows from \eqref{lem_L2stab_eq01} that
\begin{equation}
\label{bh_l2_bound}
\| \nabla v_h \|^2_{0,\Omega} \lesssim b_h(\nabla v_h,\nabla v_h), ~~~ \forall v_h \in V^n_h,
\end{equation}
which immediately indicates that the decomposition in \eqref{Hdeomp1} is well-defined.
\begin{remark}
\label{rem_decomposition}
As $b_h(\cdot,\cdot)\approx (\cdot,\cdot)_{\Omega}$, replacing the standard $L^2$ inner product $(\cdot,\cdot)_{\Omega}$ by $b_h(\cdot,\cdot)$ is reasonable, 
which also greatly facilitates the analysis below. 
In addition, \eqref{Hdeomp1} can be rewritten as given $\bfv_h\in \bfV^e_h$ finding $q_h\in V^n_h$ such that
\begin{equation}
\label{rem_decomposition_eq1}
b_h(\nabla q_h, \nabla s_h) = b_h(\bfv_h, \nabla s_h), ~~~ \forall s_h \in V^n_h,
\end{equation}
which is basically the VEM for the $H^1$ equation with the discontinuous coefficients. 
Even though the definition of virtual spaces has been modified, as the projection remains unchanged, 
the VEM here is exactly the one in the literature \cite{2013BeiraodeVeigaBrezziCangiani,2014VeigaBrezziMariniRusso,2018CaoChen}.
Therefore, the analysis of the present work can be extended to the $H^1$ equation easily, which is actually simpler due to the coercivity.
$\Box$
\end{remark}

We define the discrete div-free subspace of $\bfV^e_h$:
\begin{equation}
\label{dis_discrete}
\bfK_h(\beta_h;\Omega) = \{ \bfw_h\in \bfV^e_h~:~ b_h( \bfw_h, \nabla s_h) = 0, ~~ \forall s_h \in V^n_h \}. 
\end{equation}
The following lemma shows that the approximation relationship between $\bfK_h(\beta_h;\Omega)$ and $\bfK(\beta;\Omega)$ can be inherited.
\begin{lemma}
\label{lem_Hodge_est}
Under \hyperref[asp:I1]{(I1)} in Assumption \ref{assum_geo} and the mesh Assumption \ref{interf_eps} or \ref{assump_vmesh}.
Given any $\bfw_h \in \bfK_h(\beta_h;\Omega)$, $\mathcal{H}_{\beta}\bfw_h\in \bfK(\beta;\Omega)$ satisfies 
\begin{subequations}
\label{lem_Hodge_est_eq0}
\begin{align}
   & \|  \mathcal{H}_{\beta}\bfw_h  \|_{\theta,\Omega^{\pm}} \lesssim \| \rot \bfw_h \|_{0,\Omega}, \label{lem_Hodge_est_eq01}  \\
    &  \| \mathcal{H}_{\beta}\bfw_h - \bfw_h \|_{b} \lesssim h^{\theta} \| \rot \bfw_h \|_{0,\Omega}.  \label{lem_Hodge_est_eq02}
\end{align}
\end{subequations}
\end{lemma}
\begin{proof}
Due to the $\beta$-div-free property, \eqref{lem_Hodge_est_eq01} immediately follows from Lemma \ref{lem_embd_1} with the rot condition in \eqref{hodge_map}.
We concentrate on \eqref{lem_Hodge_est_eq02}. 
Define $\bfv_h := \bfw_h - \mathcal{H}_{\beta}\bfw_h = (\bfw_h - I^e_h \mathcal{H}_{\beta}\bfw_h)+(I^e_h \mathcal{H}_{\beta}\bfw_h - \mathcal{H}_{\beta}\bfw_h):= \bfv^{(1)}_h + \bfv^{(2)}_h$.
Then, we have
\begin{equation}
\begin{split}
\label{lem_Hodge_est_eq5}
 ( \bfw_h - \mathcal{H}_{\beta}\bfw_h , \bfw_h - \mathcal{H}_{\beta}\bfw_h )_{b}  
 =  \underbrace{(  \bfw_h , \bfv^{(1)}_h )_{b} }_{(I)} - \underbrace{ (   \mathcal{H}_{\beta}\bfw_h , \bfv^{(1)}_h )_{b}}_{(II)}  + \underbrace{ (  \bfw_h - \mathcal{H}_{\beta}\bfw_h , \bfv^{(2)}_h )_{b} }_{(III)}.
\end{split}
\end{equation}
We proceed to estimate each term above.
The commutative diagram in \eqref{DR_curl} implies
\begin{equation}
\label{lem_Hodge_est_eq3}
\rot \bfv^{(1)}_h = \rot( \bfw_h - I^e_h\mathcal{H}_{\beta}\bfw_h ) = \rot I^e_h( \bfw_h - \mathcal{H}_{\beta}\bfw_h ) = \Pi_h \rot( \bfw_h - \mathcal{H}_{\beta}\bfw_h ) =0.
\end{equation}
Then, by the exact sequence, $\bfv^{(1)}_h = \bfw_h - I^e_h\mathcal{H}_{\beta}\bfw_h\in \nabla V^n_h\subset \nabla H^1_0(\Omega)$, and thus the condition of $\bfw_h$ shows that $(I)$ vanishes. 
As for $(II)$, using the $\beta$-div-free property of $\mathcal{H}_{\beta}\bfw_h$ again, we have $(\beta \mathcal{H}_{\beta}\bfw_h, \bfv_h^{(1)})_{\Omega}=0$, and thus obtain
\begin{equation*}
\begin{split}
\label{lem_Hodge_est_eq4}
( \mathcal{H}_{\beta}\bfw_h, \bfv^{(1)}_h)_{b} & = \sum_{K\in\mathcal{T}_h} \underbrace{ ( \beta_h \Pi_K \mathcal{H}_{\beta}\bfw_h - \beta \mathcal{H}_{\beta}\bfw_h,  \bfv^{(1)}_h)_{K} }_{(IIa)}  +  \underbrace{ h ( (I-\Pi_K)\mathcal{H}_{\beta}\bfw_h\cdot\bft, (I-\Pi_K)\bfv^{(1)}_h\cdot\bft )_{\partial K} }_{(IIb)} .
\end{split}
\end{equation*}
By Lemma \ref{lem_Pi_est}, \eqref{lem_L2stab_eq01} in Lemma \ref{lem_L2stab} and Assumption \ref{interf_eps} with Lemma \ref{lem_strip}, we have
\begin{equation}
\begin{split}
\label{lem_Hodge_est_eq4_1}
 \sum_{K\in\mathcal{T}_h}  (IIa) & = \sum_{K\in\mathcal{T}_h}( \beta_h \Pi_K \mathcal{H}_{\beta}\bfw_h - \beta_h \mathcal{H}_{\beta}\bfw_h,  \bfv^{(1)}_h)_{K} + ( (\beta_h - \beta) \mathcal{H}_{\beta}\bfw_h,  \bfv^{(1)}_h)_{K} \\
& \le h^{\theta} \| \mathcal{H}_{\beta}\bfw_h \|_{\theta,\Omega^{\pm}} \| \bfv^{(1)}_h \|_{0,\Omega} + h  \| \mathcal{H}_{\beta}\bfw_h \|_{\theta,\Omega^{\pm}} \| \bfv^{(1)}_h \|_{0,\Omega}  \lesssim h^{\theta} \| \mathcal{H}_{\beta}\bfw_h \|_{\theta,\Omega^{\pm}} \| \bfv^{(1)}_h \|_b,
\end{split}
\end{equation}
where $\| \bfv^{(1)}_h \|_{b} = \| \bfv^{(1)}_h \|_a$. 
The estimate of $(IIb)$ follows from Lemma \ref{lem_Pi_est} immediately:
\begin{equation}
\label{lem_Hodge_est_eq4_2}
 \sum_{K\in\mathcal{T}_h}  (IIb) \lesssim h^{\theta} \| \mathcal{H}_{\beta}\bfw_h \|_{\theta,\Omega^{\pm}} \| \bfv^{(1)}_h \|_b.
\end{equation}
The estimate of $(III)$ follows from applying \eqref{lem_interp_est_eq02} in Lemma \ref{lem_interp_est} to $\bfv^{(2)}_h$.
Therefore, \eqref{lem_Hodge_est_eq5} becomes
\begin{equation*}
\begin{split}
\label{lem_Hodge_est_eq6}
\| \bfv_h \|^2_{b} &\lesssim h^{\theta} \| \mathcal{H}_{\beta}\bfw_h \|_{\theta,\Omega^{\pm}} ( \| \bfv_h \|_{b} + \| \bfv^{(2)}_h \|_{b}) + \| \bfv_h  \|_{b}\| \bfv^{(2)}_h \|_{b} \\
& \lesssim h^{\theta} ( \| \mathcal{H}_{\beta}\bfw_h \|_{\theta,\Omega^{\pm}} + \| \rot \mathcal{H}_{\beta}\bfw_h \|_{0,\Omega} ) \| \bfv_h \|_{b}  + h^{2\theta}\| \mathcal{H}_{\beta}\bfw_h \|_{\theta,\Omega^{\pm}} ( \| \mathcal{H}_{\beta}\bfw_h \|_{\theta,\Omega^{\pm}} + \| \rot \mathcal{H}_{\beta}\bfw_h \|_{0,\Omega} ),
\end{split}
\end{equation*}
which yields the desired estimate with \eqref{lem_Hodge_est_eq01}.
\end{proof}

 

  \section{Error estimates}
  \label{sec:erroreqn}
  
  Now, we proceed to estimate the solution errors for the virtual scheme \eqref{vem0}, and point out again that the main difficult is the relatively low regularity of the solution. 
Based on Lemma \ref{lem_embd_1}, we shall consider the solutions to \eqref{model1} in the following space
\begin{equation}
\label{solureg}
\bfX^{\theta}(\Omega) = \{ \bfu\in \bfH^{\theta}(\Omega^{\pm}): \alpha\rot\bfu\in \bfH^{\theta}(\Omega) \}, 
\end{equation}
where $\theta$ is inherited from Lemma \ref{lem_embd_1}.
However, such a regularity requirement alone cannot lead us to the optimal error estimates in that a straightforward application of the meta VEM framework in the literature only yields a suboptimal convergence.
In order to overcome this issue, we shall make the following assumption regarding the regularity of the source term:
 \begin{equation}
\label{fregularity}
\bff \in \bfH^{\theta'}(\Omega^{\pm}) , ~~~ \text{for some} ~ \theta' \ge \theta .
\end{equation}
Consequently, the following stability holds:
\begin{equation}
\label{fregularity2}
\| \bfu \|_{\theta,\rot;\Omega} \lesssim \| \bff \|_{0,\Omega} \le \| \bff \|_{\theta',\Omega^{\pm}}.
\end{equation}
 We highlight that this is merely the regularity for the source not for the exact solution. Generally, the source regularity alone cannot be directly related to the solution regularity due to the appearance of discontinuous coefficients and geometric singularities, see \cite{2000CostabelDauge,1999MartinMoniqueSerge,2004CostabelDaugeNicaise}. 
 In addition, assuming higher regularity for the source term is often used for obtaining optimal error estimates when the solution regularity is pretty low, see \cite{2016Ciarlet,2002Nicaise}.
 

Now, we proceed to present the error analysis of the indefinite scheme.
First decompose the solution error $\bfe_h := \bfu - \bfu_h= \bfxi_h + \bfeta_h$ where 
  \begin{equation}
\label{err_decomp_1}
\bfxi_h =  \bfu - I^e_h \bfu ~~~~ \text{and} ~~~~ \bfeta_h = I^e_h \bfu - \bfu_h.
\end{equation}
The difficulty is to estimate $\bfeta_h$, and the idea is to employ its discrete decomposition in terms of the inner product $b_h(\cdot,\cdot)$ defined in \eqref{Hdeomp1}:
\begin{equation}
\label{err_decomp_2}
\bfeta_h = \bfr_h + \nabla q_h, ~~~ \bfr_h\in \bfK_h(\beta_h;\Omega), ~~q_h\in V^n_h.
\end{equation}
Hence, the error estimation of $\bfeta_h$ reduces to estimating $\bfr_h$ and $\nabla q_h$.
For $\| \bfe_h \|_{0,\Omega}$, the following lemma shows that it is sufficient to bound $\|\bfe_h\|_a$.
\begin{lemma}
\label{lem_L2}
Let $\bfu\in \bfX^{\theta}(\Omega)$. Then, there holds
\begin{equation}
\label{lem_L2_eq0}
\| \bfe_h \|_{0,\Omega} \lesssim h^{\theta}\| \bff \|_{\theta',\Omega^{\pm}} + \| \bfe_h \|_a .
\end{equation} 
\end{lemma}
\begin{proof}
Due to \eqref{lem_L2stab_eq01} in Lemma \ref{lem_L2stab}, we have $\| \bfeta_h \|_{0,\Omega} \lesssim \| \bfeta_h \|_{a}$ as $\bfeta_h \in \bfV^e_h$.
Then, by Lemma \ref{lem_interp_est} with the triangular inequality, we obtain
\begin{equation}
\label{lem_L2_eq1}
\| \bfe_ h \|_{0,\Omega}  \lesssim \| \bfxi_h \|_{0,\Omega} + \| \bfeta_h \|_{a} \lesssim \| \bfxi_h \|_{0,\Omega} + \| \bfxi_h \|_a + \| \bfe_h \|_{a} \lesssim h^{\theta} \| \bfu \|_{\theta,\rot;\Omega} + \| \bfe_h \|_a,
\end{equation}
which yields \eqref{lem_L2_eq0} by \eqref{fregularity2}.
\end{proof}

Note that the virtual element solution $\bfu_h$ is not consistent to the true solution either in terms of either the bilinear form $a(\cdot,\cdot)$ or $a_h(\cdot,\cdot)$. 
But, we can show that the inconsistency errors have the optimal order, which is given by the following lemma.
\begin{lemma}
\label{lem_consist}
Let $\bfu\in \bfX^{\theta}(\Omega)$ and let $\bff\in \bfH^{\theta'}(\Omega^{\pm}) $.
There holds that
\begin{equation}
\label{lem_consist_eq0}
a_h(\bfe_h,\bfv_h) \lesssim h^{\theta} \| \bff \|_{\theta',\Omega^{\pm}}  \| \bfv_h \|_{a} , ~~~ \forall \bfv_h \in \bfV^e_h.
\end{equation}
\end{lemma}
\begin{proof}
Notice that
\begin{equation}
\begin{split}
\label{lem_consist_eq1}
a_h(\bfe_h,\bfv_h) & = a_h(\bfu,\bfv_h) - a_h(\bfu_h,\bfv_h) = a_h(\bfu,\bfv_h) - (\bff, \bfv_h)_{\Omega} + (\bff, \bfv_h - \Pi_h \bfv_h)_{\Omega}, \\
& =  \sum_{K\in\mathcal{T}_h}\underbrace{ (\alpha_h \rot \bfu, \rot \bfv_h)_{K} - ( \alpha \rot \bfu, \rot \bfv_h )_{K}}_{(I)} \\
& + \underbrace{ (\beta \bfu, \bfv_h)_{K} - b_h(\bfu,\bfv_h)_K }_{(II)} + \underbrace{ (  \bff, \bfv_h - \Pi_K \bfv_h )_K }_{(III)} ,
\end{split}
\end{equation} 
where in the third equality we have used integration by parts and $\bfv_h$ being $\bfH(\rot;\Omega)$-conforming.
We first estimate $(I)$ by applying Lemma \ref{lem_strip}:
\begin{equation}
\begin{split}
\label{lem_consist_eq2}
\sum_{K\in\mathcal{T}_h} (I) & \le \| \rot \bfu \|_{0,\Omega^{\Gamma}_{\epsilon h^2}} \| \rot\bfv_h \|_{0,\Omega} \lesssim h \| \rot \bfu \|_{\theta,\rot;\Omega^{\pm}} \| \bfv_h \|_a.
\end{split}
\end{equation}
 Next, for $(II)$, by \eqref{lem_L2stab_eq01} in Lemma \ref{lem_L2stab}, Lemma \ref{lem_Pi_est} and Lemma \ref{lem_strip}, we have
 \begin{equation}
 \begin{split}
\label{lem_consist_eq3}
\sum_{K\in\mathcal{T}_h}(II) & =  \sum_{K\in\mathcal{T}_h} [ \beta_h (\bfu -   \Pi_K \bfu) ,  \bfv_h)_{K} + ( \beta \bfu -  \beta_h  \bfu ,  \bfv_h)_{K}  \\
& ~~~~~~~~~~~~ -  h_K ( (I - \Pi_K)\bfu\cdot\bft , (I - \Pi_K)\bfv_h\cdot\bft  )_{\partial K} ] \\
 & \lesssim \| \bfu \|_{0,\Omega^{\Gamma}_{\epsilon h^2}} \| \bfv_h \|_{0,\Omega}  + h^{\theta} \| \bfu \|_{\theta,\rot,\Omega^{\pm}} \| \bfv_h \|_a  \lesssim h^{\theta} \| \bfu \|_{\theta,\rot,\Omega^{\pm}} \| \bfv_h \|_a  .
 \end{split}
\end{equation}
As for $(III)$, by inserting $\Pi_K \bff$ and using Lemma \ref{lem_Pi_est} and \eqref{lem_L2stab_eq01} in Lemma \ref{lem_L2stab}, we have
\begin{equation}
\begin{split}
\label{lem_consist_eq4}
& \sum_{K\in\mathcal{T}_h} (III)  = \sum_{K\in\mathcal{T}_h}  ( \bff - \Pi_K \bff, \bfv_h - \Pi_K \bfv_h )_K \\
 \lesssim & \left(\sum_{K\in\mathcal{T}_h} \| \bff - \Pi_K \bff \|^2_{0,K} \right)^{1/2} \left( \sum_{K\in\mathcal{T}_h} \|  \bfv_h - \Pi_K \bfv_h  \|^2_{0,K} \right)^{1/2}
\lesssim h^{\theta'} \| \bff \|_{\theta',\Omega^{\pm}}  \| \bfv_h  \|_a.
\end{split}
\end{equation}
Putting \eqref{lem_consist_eq2}-\eqref{lem_consist_eq4} into \eqref{lem_consist_eq1} yields the desired result with \eqref{fregularity2}.
\end{proof}

Then, we consider the consistence error associated with the original bilinear form $a(\cdot,\cdot)$, and a similar estimate for the $H^1$-elliptic equation can be found in \cite{2018BrennerSung}.
\begin{lemma}
\label{lem_consist_a}
Let $\bfu,\bfv\in \bfX^{\theta}(\Omega)$ and let $\bff\in \bfH^{\theta'}(\Omega^{\pm}) $.
There holds that
\begin{equation}
\label{lem_consist_a_eq0}
a(\bfe_h, I^e_h \bfv) \lesssim h^{\theta} (  \| \bff \|_{\theta',\Omega^{\pm}} + \| \bfu_h \|_{b} ) \| \bfv \|_{\theta,\rot;\Omega^{\pm}} .
\end{equation}
\end{lemma}
\begin{proof}
We notice that
\begin{equation}
\begin{split}
\label{lem_consist_a_eq1}
a(\bfe_h, \bfv_h ) & = \underbrace{( (\alpha \rot \bfu, \rot I^e_h\bfv)_{\Omega}  - (\beta \bfu, I^e_h\bfv)_{\Omega} )}_{(I)} \\
 &  - \underbrace{ ( (\alpha \rot \bfu_h, \rot I^e_h\bfv)_{\Omega}  - b_h(\bfu_h, I^e_h\bfv) ) }_{(II)}  - \underbrace{ ( b_h(\bfu_h, I^e_h\bfv) - (\beta  \bfu_h, I^e_h\bfv )_{\Omega}) }_{(III)}.
\end{split}
\end{equation}
By \eqref{vem0}, we can write down
\begin{equation}
\begin{split}
\label{lem_consist_a_eq2}
(II) & = (\alpha_h \rot \bfu_h, \rot I^e_h\bfv )_{\Omega}  - b_h(\bfu_h, I^e_h\bfv) + ( (\alpha-\alpha_h) \rot \bfu_h, \rot I^e_h\bfv)_{\Omega} \\
& = (\bff, \Pi_h I^e_h\bfv )_{\Omega} + ( (\alpha-\alpha_h) \rot \bfu_h, \rot \bfv)_{\Omega}  + ( (\alpha-\alpha_h) \rot \bfu_h, \rot (I^e_h\bfv - \bfv) )_{\Omega}.
\end{split}
\end{equation}
As $(I)=(\bff, I^e_h\bfv)_{\Omega}$, by using the similar argument to $(III)$ in \eqref{lem_consist_eq1}, \eqref{lem_interp_est_eq03} in Lemma \ref{lem_interp_est} and the thin-strip argument by Lemma \ref{lem_strip}, we obtain
\begin{equation}
\begin{split}
\label{lem_consist_a_eq3}
 (I) - (II)  = & (\bff , I^e_h\bfv - \Pi_h I^e_h\bfv)_{\Omega}  -  ( (\alpha-\alpha_h) \rot \bfu_h, \rot \bfv)_{\Omega}  - ( (\alpha-\alpha_h) \rot \bfu_h, \rot (I^e_h\bfv - \bfv) )_{\Omega} \\
\lesssim &  h^{\theta'}  \| \bff \|_{\theta',\Omega^{\pm}} \|  I^e_h\bfv \|_{a} +  h^{\theta}  \| \rot \bfu_h \|_{0,\Omega} \| \rot\bfv \|_{\theta,\Omega^{\pm}} .
\end{split}
\end{equation}
Furthermore, it is not hard to see $\| I^e_h \bfv \|_a\lesssim \| \bfv \|_{\theta,\rot;\Omega^{\pm}}$.
Next, for $(III)$ in \eqref{lem_consist_a_eq1}, we notice
\begin{equation}
\begin{split}
\label{lem_consist_a_eq4}
(III) = & \underbrace{ (\beta_h  \bfu_h, \Pi_h I^e_h \bfv - I^e_h \bfv)_{\Omega} }_{(IIIa)} + \underbrace{ ( (\beta_h - \beta)\bfu_h, I^e_h \bfv - \bfv )_{\Omega} }_{(IIIb)}\\
& + \underbrace{ ( (\beta_h - \beta)\bfu_h,  \bfv )_{\Omega} }_{(IIIc)} - \underbrace{ h \sum_{K\in\mathcal{T}_h} ( (I-\Pi_K)\bfu_h\cdot\bft, (I-\Pi_K) I^e_h\bfv\cdot\bft)_{\partial K} }_{(IIId)} .
\end{split}
\end{equation}
The estimate of $(IIIa)$ follows from the following decomposition together with \eqref{lem_interp_est_eq01} in Lemma \ref{lem_interp_est} and Lemma \ref{lem_Pi_est}:
\begin{equation}
\label{lem_consist_a_eq5}
\| \Pi_h I^e_h \bfv - I^e_h \bfv \|_{0,\Omega} \le \| \Pi_h(I^e_h\bfv-\bfv) \|_{0,\Omega} + \| \Pi_h \bfv - \bfv \|_{0,\Omega}  + \| \bfv - I^e_h \bfv \|_{0,\Omega}  \lesssim h^{\theta} \| \bfv \|_{\theta,\rot;\Omega}.
\end{equation}
$(IIIb)$ follows from \eqref{lem_interp_est_eq01} immediately, while $(IIIc)$ can be estimated by the thin-strip argument from Lemma \ref{lem_strip}. Then, it remains to estimate $(IIId)$. The H\"older's inequality, \eqref{lem_interp_est_eq02} in Lemma \ref{lem_interp_est} and Lemma \ref{lem_Pi_est} imply
\begin{equation}
\label{lem_consist_a_eq6}
(IIId) \lesssim \| \bfu_h \|_b \left(( \| I^e_h \bfv - \bfv \|_b +  h \sum_{K\in\mathcal{T}_h} \|(I - \Pi_K)\bfv\cdot\bft \|^2_{0,\partial K}  \right) \lesssim h^{\theta} \| \bfu_h \|_b \| \bfv \|_{\theta,\rot;\Omega}.
\end{equation}
Putting the estimates above together into \eqref{lem_consist_a_eq4}, we obtain the estimate of $(III)$. Combining it with \eqref{lem_consist_a_eq3}, we have the desired estimate.
\end{proof}


Now, we are ready to estimate $\bfe_h$ and $\rot\bfe_h$ by borrowing some argument from \cite{2009ZhongShuWittumXu}.
\begin{lemma}
\label{lem_err0}
Let $\bfu\in \bfX^{\theta}(\Omega)$ and $\bff\in \bfH^{\theta'}(\Omega^{\pm}) $.
There holds that
\begin{equation}
\label{lem_err0_eq0}
\| \bfe_h \|_{b} \lesssim h^{\theta}  \| \bff \|_{\theta',\Omega^{\pm}}   +  \| \bfr_h \|_b.
\end{equation}
\end{lemma}
  \begin{proof}
We can apply H\"older's inequality to obtain
  \begin{equation}
  \begin{split}
\label{lem_err0_eq1}
b_h(\bfe_h,\bfe_h) & = b_h(\bfe_h, \bfxi_h) +  b_h(\bfe_h, \bfr_h) +  b_h(\bfe_h, \nabla q_h) \\
& \leqslant \| \bfxi_h \|^2_b + \| \bfr_h \|^2_b + \| \bfe_h \|^2_b/2 + |b_h(\bfe_h, \nabla q_h) |.
\end{split}
\end{equation}
The estimate of $\bfxi_h$ is already given by \eqref{lem_interp_est_eq02} in Lemma \ref{lem_interp_est}.
We only need to estimate the last term on the right-hand side of \eqref{lem_err0_eq1}. By Lemma \ref{lem_consist}, as $a_h(\bfe_h,\nabla q_h) = b_h(\bfe_h,\nabla q_h) $, we have
\begin{equation}
\begin{split}
\label{lem_err0_eq2}
 |b_h(\bfe_h, \nabla q_h)| &\lesssim h^{\theta'}  \| \bff \|_{\theta',\Omega^{\pm}}    \| \nabla q_h \|_b  \le  h^{2\theta'}  \| \bff \|^2_{\theta',\Omega^{\pm}}   + ( \| \bfe_h \|^2_b + \| \bfxi_h \|^2_b + \| \bfr_h \|^2_b)/4.
 \end{split}
\end{equation} 
Putting \eqref{lem_err0_eq2} into \eqref{lem_err0_eq1}, and applying the stability \eqref{fregularity2}, we arrive at the desired estimate.
  \end{proof}

  \begin{lemma}
\label{lem_errcurl}
Let $\bfu\in \bfX^{\theta}(\Omega)$ and $\bff\in \bfH^{\theta'}(\Omega^{\pm}) $.
There holds that
\begin{equation}
\label{lem_errcurl_eq0}
\| \rot \bfe_h \|_{0,\Omega} \lesssim h^{\theta}  \| \bff \|_{\theta',\Omega^{\pm}}   +  \| \bfr_h \|_b.
\end{equation}
\end{lemma}
\begin{proof}
Consider the following identity from \eqref{err_decomp_2}:
\begin{equation}
\begin{split}
\label{lem_errcurl_eq1}
\| \alpha^{1/2}_h \rot \bfe_h \|^2_0 &= a_h(\bfe_h,\bfe_h) + \| \bfe_h \|^2_{b}  = a_h(\bfe_h, \bfxi_h) + a_h(\bfe_h,\bfeta_h) + b_h(\bfe_h,\bfxi_h+\bfeta_h) \\
& = (\alpha^{1/2}_h \rot \bfe_h, \rot \bfxi_h)_{\Omega} +  a_h(\bfe_h,\bfeta_h)  + b_h(\bfe_h,\bfr_h)  + b_h(\bfe_h, \nabla q_h).
\end{split}
\end{equation}
Lemma \ref{lem_consist} yields
\begin{subequations}
\label{lem_errcurl_eq2}
\begin{align}
    &  a_h(\bfe_h,\bfeta_h) \lesssim  h^{\theta} \| \bff \|_{\theta',\Omega^{\pm}}    ( \| \rot \bfe_h \|_{0,\Omega} + \| \bfe_h \|_{b} + \| \bfxi_h \|_{a} ). \\
    &  b_h(\bfe_h, \nabla q_h) = a_h(\bfe_h, \nabla q_h)  \lesssim  h^{\theta} \| \bff \|_{\theta',\Omega^{\pm}}  ( \| \bfe_h \|_b + \| \bfxi_h \|_b + \| \bfr_h \|_b  ).
\end{align}
\end{subequations}
In addition, $b_h(\bfe_h,\bfr_h) \le \| \bfe_h \|_b \| \bfr_h \|_b$. Substituting this estimate and \eqref{lem_errcurl_eq2} into \eqref{lem_errcurl_eq1}, applying the similar argument to \eqref{lem_err0_eq2}, and using Lemma \ref{lem_err0}, we have \eqref{lem_errcurl_eq0}.
\end{proof}
 
   Now, based on the Lemmas \ref{lem_errcurl} and \ref{lem_err0}, it remains to estimate $\bfr_h$. 
   \begin{lemma}
   \label{lem_rh}
 Let $\bfu\in \bfX^{\theta}(\Omega)$ and let $\bff\in \bfH^{\theta'}(\Omega^{\pm})$.  There holds that
   \begin{equation}
\label{lem_rh_eq0}
\| \bfr_h \|_b  \lesssim h^{\theta}\| \bff \|_{\theta',\Omega^{\pm}}. 
\end{equation}
   \end{lemma}
   \begin{proof}
 We consider the Hodge mapping defined in \eqref{hodge_map}. As $\bfr_h\in \bfK_h(\beta;\Omega)$, from \eqref{lem_interp_est_eq03} in Lemma \ref{lem_interp_est}, and using Lemmas \ref{lem_Hodge_est} and \ref{lem_errcurl}, we have
   \begin{equation}
   \begin{split}
\label{lem_rh_eq1}
\| \bfr_h - \mathcal{H}_{\beta} \bfr_h \|_{b}& \lesssim h^{\theta} \| \rot \bfr_h \|_{0,\Omega} =  h^{\theta} \left( \| \rot \bfe_h \|_{0,\Omega} + \| \rot \bfxi_h \|_{0,\Omega}  \right)   \le  h^{2\theta}  \| \bff \|_{\theta',\Omega^{\pm}}    +   h^{\theta}  \| \bfr_h \|_b  .
\end{split}
\end{equation}
Thus, there exists a constant $C$ such that
\begin{equation}
\label{lem_rh_eq2}
(1 - C h^{\theta})\| \bfr_h \|_b \lesssim h^{2\theta}  \| \bff \|_{\theta',\Omega^{\pm}}   + \| \mathcal{H}_{\beta} \bfr_h \|_{b}.
\end{equation}
We now use a duality argument to estimate $\| \mathcal{H}_{\beta} \bfr_h \|_{b}$. Consider the following auxiliary problem of $\bfw \in \bfH_0(\rot;\Omega)\cap\bfH(\beta,\ddiv;\Omega)$:
\begin{equation}
\label{lem_rh_eq3}
\brot \alpha \rot \bfw - \beta \bfw = \beta \mathcal{H}_{\beta} \bfr_h.
\end{equation}
As the problem has the same coefficients as the model problem, the following stability holds:
\begin{equation}
\label{lem_rh_eq3_1}
\| \bfw \|_{\theta,\rot,\Omega^{\pm}} \lesssim \| \mathcal{H}_{\beta} \bfr_h \|_{0,\Omega}.
\end{equation}
Testing \eqref{lem_rh_eq3} by $\bfv\in\bfH_0(\rot;\Omega)$, we have
\begin{equation}
\label{lem_rh_eq4}
a(\bfw,\bfv) = (\alpha\rot \bfw, \rot \bfv)_{\Omega} - (\beta \bfw, \bfv)_{\Omega} = (\beta \mathcal{H}_{\beta} \bfr_h,\bfv)_{\Omega}.
\end{equation}
Then, as $\mathcal{H}_{\beta} \bfr_h\subset\bfK(\beta;\Omega)$ given in \eqref{div_free}, there holds
\begin{equation}
\label{lem_rh_eq5}
a(\bfw,\nabla q) = (\beta \bfw, \nabla q)_{\Omega} = 0, ~~~ \forall q \in H^1_0(\Omega).
\end{equation}
As $\rot(\mathcal{H}_{\beta} \bfr_h) = \rot(\bfr_h) = \rot(\bfeta_h)$, by the exact sequence, $\mathcal{H}_{\beta} \bfr_h -\bfeta_h \in \nabla H^1_0(\Omega)$. So, \eqref{lem_rh_eq5} implies
\begin{equation}
\label{lem_rh_eq6}
a(\bfw,\mathcal{H}_{\beta} \bfr_h -\bfeta_h) = (\beta \bfw, \mathcal{H}_{\beta} \bfr_h -\bfeta_h)_{\Omega} = 0.
\end{equation}
In addition, taking $\bfv =\mathcal{H}_{\beta} \bfr_h $ in \eqref{lem_rh_eq4} and using \eqref{lem_rh_eq6} yields
\begin{equation}
\begin{split}
\label{lem_rh_eq7}
\| \sqrt{\beta} \mathcal{H}_{\beta} \bfr_h \|^2_{0,\Omega} & = a(\bfw, \mathcal{H}_{\beta} \bfr_h) = a(\bfw, \mathcal{H}_{\beta} \bfr_h - \bfeta_h) + a(\bfw, \bfeta_h) \\
& = a(\bfw - I^e_h \bfw, \bfe_h ) + a(I^e_h \bfw, \bfe_h ) - a(\bfw, \bfxi_h) := (I) + (II) + (III).
\end{split}
\end{equation}
We proceed to estimate each term above.
For $(I)$, \eqref{lem_interp_est_eq01} and \eqref{lem_interp_est_eq03} in Lemma \ref{lem_interp_est} imply
\begin{equation}
\label{lem_rh_eq8}
(I) \lesssim \| \bfe_h \|_{\rot,\Omega}  \| \bfw - I^e_h \bfw \|_{\rot,\Omega} \lesssim h^{\theta} \| \bfw \|_{\theta,\rot;\Omega^{\pm}}  \| \bfe_h \|_{\rot,\Omega}.
\end{equation}
The estimate of $(II)$ follows from Lemma \ref{lem_consist_a}, while $(III)$ follows from Lemma \ref{lem_interp_est} again.
Putting these estimates into \eqref{lem_rh_eq7}, using \eqref{lem_rh_eq3_1} and cancelling $\|\mathcal{H}_{\beta}\bfr_h\|_{0,\Omega}$ in \eqref{lem_rh_eq10} leads to 
\begin{equation}
\begin{split}
\label{lem_rh_eq10}
\|  \mathcal{H}_{\beta} \bfr_h \|_{0,\Omega} 
 \lesssim h^{\theta} ( \| \bff \|_{\theta';\Omega^{\pm}} +  \| \bfu_h \|_{b}  + \| \bfe_h \|_{\rot;\Omega}  ). 
\end{split}
\end{equation}
Here, special attention needs to be paid to $\| \bfe_h \|_{\rot,\Omega}$ which should be bounded in terms of $\| \bfe_h\|_{a}$. Note that $\bfe_h = \bfxi_h + \bfeta_h$, and as $\bfeta_h \in \bfV^e_h$, by \eqref{lem_L2stab_eq01} in Lemma \ref{lem_L2stab} and Lemma \ref{lem_interp_est} there holds
\begin{equation}
\label{lem_rh_eq10_1}
\| \bfe_h \|_{0,\Omega} \lesssim \| \bfxi_h \|_{0,\Omega} + \| \bfeta_h \|_{a} \lesssim  h^{\theta}  \| \bff \|_{\theta';\Omega^{\pm}}  + \| \bfe_h \|_a.
\end{equation}
In addition, we have $\| \bfu_h \|_b \le \| \bfe_h \|_b + \| \bfu \|_b$. The estimate of $\| \bfu \|_b$ follows from \eqref{lem_L2stab_eq02} in Lemma \ref{lem_L2stab}. We then bound $\| \bfe_h \|_{b} $ and $\| \rot \bfe_h \|_{0,\Omega} $ by Lemmas \ref{lem_err0} and \ref{lem_errcurl}. Hence, \eqref{lem_rh_eq10} becomes $\| \mathcal{H}_{\beta} \bfr_h \|_{0,\Omega}  \lesssim h^{\theta} ( \| \bff \|_{\theta';\Omega^{\pm}} + \| \bfr_h \|_b )$. Applying \eqref{lem_L2stab_eq02} in Lemma \ref{lem_L2stab} and \eqref{lem_Hodge_est_eq01} in Lemma \ref{lem_Hodge_est}, we achieve
\begin{equation}
\begin{split}
\label{lem_rh_eq11}
\|  \mathcal{H}_{\beta} \bfr_h \|_{b} & \lesssim \| \mathcal{H}_{\beta} \bfr_h \|_{0,\Omega} + h^{\theta} \| \rot \bfr_h \|_{0,\Omega}  \lesssim h^{\theta} ( \| \bff \|_{\theta';\Omega^{\pm}} + \| \bfr_h \|_b + \| \rot \bfr_h \|_{0,\Omega} ) \\
& \lesssim h^{\theta} ( \| \bff \|_{\theta';\Omega^{\pm}} + \| \bfr_h \|_b + \| \rot \bfe_h \|_{0,\Omega} + \| \rot \bfxi_h \|_{0,\Omega} )   \lesssim h^{\theta} ( \| \bff \|_{\theta';\Omega^{\pm}} + \| \bfr_h \|_b )
\end{split}
\end{equation}
where we have also used Lemma \ref{lem_errcurl} and \eqref{lem_interp_est_eq03} in the last inequality.
Hence,  from \eqref{lem_rh_eq2}, assuming that $h$ is sufficiently small, using $h^{\theta'}\lesssim 1$, we obtain the desired result.
  \end{proof}
 Finally, we can achieve the optimal error estimates in the following lemma.
 \begin{lemma}
 \label{lem_error}
  Let $\bfu\in \bfX^{\theta}(\Omega)$ and let $\bff\in \bfH^{\theta'}(\Omega^{\pm})$.  Then, there holds
 \begin{equation}
\label{lem_error_eq0}
\| \bfe_h \|_{\rot;\Omega} \lesssim h^{\theta} \| \bff \|_{\theta',\Omega^{\pm}}.
\end{equation}
 \end{lemma}


  \section{Proofs of Lemmas \ref{lem_Pi_est} and \ref{lem_interp_est}}
  \label{sec:lemmaproof}

In this section, we verify Lemmas \ref{lem_Pi_est} and \ref{lem_interp_est} under the geometric assumptions specified in Subsection \ref{subsec:geo}. 
The following approximation result regarding the projection operator on general domain will be frequently used.
\begin{lemma}
\label{lem_proj_approxi}
Let $D$ be an open domain. For all $v\in H^{s}(D)$ with $s\in(0,1)$ and $p\ge 1$, there holds
\begin{align}
\label{lem_proj_approxi_eq01}
   \| v - \Pi_D v \|_{L^2(D)} \le h^{s+1}_D |D|^{-1/2} |v|_{s,D}.
\end{align}
\end{lemma}
\begin{proof}
See \cite[Lemma 7.1]{2017ErnGuermond} for \eqref{lem_proj_approxi_eq01}.
\end{proof}
We also recall the following theorem from \cite{2015Zhou} regarding extensions of fractional Sobolev spaces.
\begin{theorem}
\label{thm_extension}
For any domain $D$ being regular in the sense of \eqref{Dregular}, there exists a continuous extension operator 
$$
\mathcal{E}_D : H^s(D) \longrightarrow H^s(\mathbb{R}^2) .
$$
\end{theorem}
With this theorem and the condition \hyperref[asp:I2]{(I2)} in Assumption \ref{assum_geo}, we can construct $\mathcal{E}^{\pm}:=\mathcal{E}_{\Omega^{\pm}}$. 
For the simplicity of presentation, for an arbitrary domain $D\subseteq \Omega$, we introduce
\begin{equation}
\label{Enorm}
\| \bfu \|_{E,s,D} = \|\mathcal{E}^+ \bfu \|_{s,D} + \| \mathcal{E}^- \bfu \|_{s,D} ~~~~ \text{satisfying} ~~ \| \bfu \|_{E,s,\Omega} \lesssim \| \bfu \|_{s,D}.
\end{equation}

 From now on we shall ignore the very thin region near the interface on which $\alpha_h$ and $\beta_h$ may be different from $\alpha$ and $\beta$, shown in Figure \ref{fig:interfelem}.
We can effectively address this issue using the thin strip argument given by Assumption \ref{interf_eps} and Lemma \ref{lem_strip}. 
For simplicity's sake, we will not bring this issue into our analysis below. 

\subsection{Geometric assumptions}
\label{subsec:geo}

To be able to establish a robust analysis of the discretization, we make the following assumptions. 
\begin{assumption}
\label{assump_polygon}
Each element $K$ in the global mesh $\mathcal{T}_h$ satisfies:
\begin{itemize}
 \item[(G1)] \label{asp:polygonG1} \textbf{A star-convexity condition}: 
 $K$ is star convex to a ball of a radius $\rho_K$ satisfying that
 \begin{equation}
\label{rho_cont}
\rho_K \ge \tau(\theta) h_K, ~~~ \text{with} ~~ \tau(\theta) = \exp( 1 + \kappa_0 - r(\theta) /(2(1-\theta)) ),
\end{equation}
for a suitable constant $\kappa_0<0$, where $r(\theta)$ is a suitable function satisfying $r(\theta)\ge 0$, $\forall \theta \in(0,1)$ and $\lim_{\theta\rightarrow 1/2}r(\theta) = \lim_{\theta\rightarrow1}r(\theta) = 1$ which, henceforth, is specifically chosen as
\begin{equation}
\label{rho_cont_1}
r(\theta) = 1 - \kappa_1 (\theta-0.5)(\theta-1)
\end{equation}
for a constant suitable $\kappa_1>0$. We refer readers to Remark \ref{rem_assum_polygon1} for the specific choices of $\kappa_0$ and $\kappa_1$ as well as the reasoning.
There are other acceptable options for $r(\theta)$, also see Remark \ref{rem_assum_polygon1}.
Here, $\theta$ is the regularity order of the solution.

 \item[(G2)] \label{asp:polygonG2} \textbf{An edge condition}: the number of edges in $\mathcal{E}_K$ is uniformly bounded. For each edge $e\in \mathcal{E}_K$, define $l_e = \max_{\bfx \in K}\dist(\bfx,e)$ as the supporting height of $e$, and suppose it satisfies one of the following two conditions:
 \begin{subequations}
\label{edge_cond}
\begin{align}
    & h_eh_K \le c_1 |K|,  \label{edge_cond_eq1}  \\
    &  c^{-1}_2 h_K \le  h_e  \le c_2 l^{-1}_e |K|, \label{edge_cond_eq2}
\end{align} 
\end{subequations}
for some uniform constants $c_1$ and $c_2$.
 \item[(G3)] \label{asp:polygonG3} \textbf{A ``block" condition}: define the block $\mathcal{B}_K$ of $K$ as the ball with a diameter $3h_K$ and with a center at the centroid of $K$; and assume $|\{ K': K'\cap \mathcal{B}_K\neq\emptyset  \}|$ is uniformly bounded for all $K$.
 \end{itemize}

\end{assumption}

\begin{remark}
\label{rem_assum_polygon1}
In this remark, we make detailed explanations of Assumption \ref{assump_polygon}.
\begin{itemize}
\item \hyperref[asp:polygonG1]{(G1)} explicitly specifies the dependence of the shape regularity parameter $\tau_{\theta}$ on the solution regularity parameter $\theta$. Remarkably, $\tau_\theta = 0$ for $\theta=1$, meaning that only star convexity to a point is required, and the element can shrink to a segment in an arbitrary manner. 
However, if the solution regularity is even lower, i.e., $\theta<1$, we need to accordingly enhance the shape-regularity of the element $K$ by \eqref{rho_cont} in order to obtain robust error estimates. 
Particularly, in the extreme case that $\theta\rightarrow 1/2$, as $\tau(1/2)=\exp(\kappa_0)$, we see that 
\begin{equation}
\label{kappa_0}
\rho_K \ge \exp(\kappa_0) h_K,
\end{equation}
meaning that $K$ has to be very shape-regular. So $\kappa_0$ is a parameter to control the shape in such a low-regularity case, and thus we typically assume $\kappa_0<0$ to ensure $\exp(\kappa_0)<1$. One can also verify that $\kappa_0<0$ is sufficient for $\tau(\theta)<1$ for any $\theta\in(1/2,1]$ and $\kappa_1>0$.

 In fact, it is worthwhile to mention that, for very irregular triangular elements merely satisfying the classical maximum angle condition, higher regularity (typically higher than $\bfH^1$) is generally needed for obtaining robust optimal error estimates, see \cite{1999AcostaRicardo,2005BuffaCostabelDauge,2011Lombardi} for example. 

\vspace{0.1in}

\item It is highlighted that most of the analysis in the literature relies on trace inequalities, and thus the resulting error bound from \eqref{rho_cont} usually inevitable involves a constant $(\tau(\theta))^{-1}$ which may blow up for very small $\tau(\theta)$. However, we shall employ a completely different and more elegant way to relax the constant to be $(\tau(\theta))^{\theta-1}$ which is uniformly bounded even for $\tau(\theta)\rightarrow 0$ with $\theta \rightarrow 1$.
This will help us obtain a much more robust and accurate error estimate.
In fact, the maximal value of $(\tau(\theta))^{\theta-1}$ with respect to $\theta\in(1/2,1]$ in terms of $\kappa_0$ and $\kappa_1$ can be explicitly expressed as 
\begin{equation}
\label{maxvc}
\varrho(\kappa_0,\kappa_1) 
= \begin{cases}
      & \exp(1/2) ~~~ \text{if}~ (\kappa_0+1)/\kappa_1 >1/4, \\
      & \exp(-\kappa_0/2) ~~~ \text{if}~ (\kappa_0+1)/\kappa_1 <-1/4, \\
      &\exp(\frac{16 \kappa_0^2 - 8 \kappa_0 (-4 + \kappa_1) + (4 + \kappa_1)^2}{32\kappa_1}) ), ~~ \text{if}~ (\kappa_0+1)/\kappa_1 \in [-1/4,1/4].
\end{cases}
\end{equation}
As increasing $\kappa_1$ can reduce $\tau(\theta)$, we shall concentrate on the third case in \eqref{maxvc}. 
Specifically, we shall fix $\kappa_0 = -1$ and $\kappa_1 = 60$, but other choices can be surely made. 
With these values, we can compute the values of $\tau(\theta)$ for various $\theta$ and report them in Figure \ref{fig:tauplot},
which are sufficient for many practical situations involving very irregular meshes. 
The resulting moderate value of $\varrho(-1,50)\approx 11$ indicates reasonable accuracy. See Figure \ref{fig:taufun} for the values of $\tau$ and $\varrho$ from other choices of $\kappa_0$ and $\kappa_1$. In summary, the choice of $\kappa_0$ and $\kappa_1$ is a critical factor for balancing the shape-regularity requirement and the constant in the error bound.



\begin{figure}[h]
\centering
\begin{subfigure}{.35\textwidth}
    \includegraphics[width=1.7in]{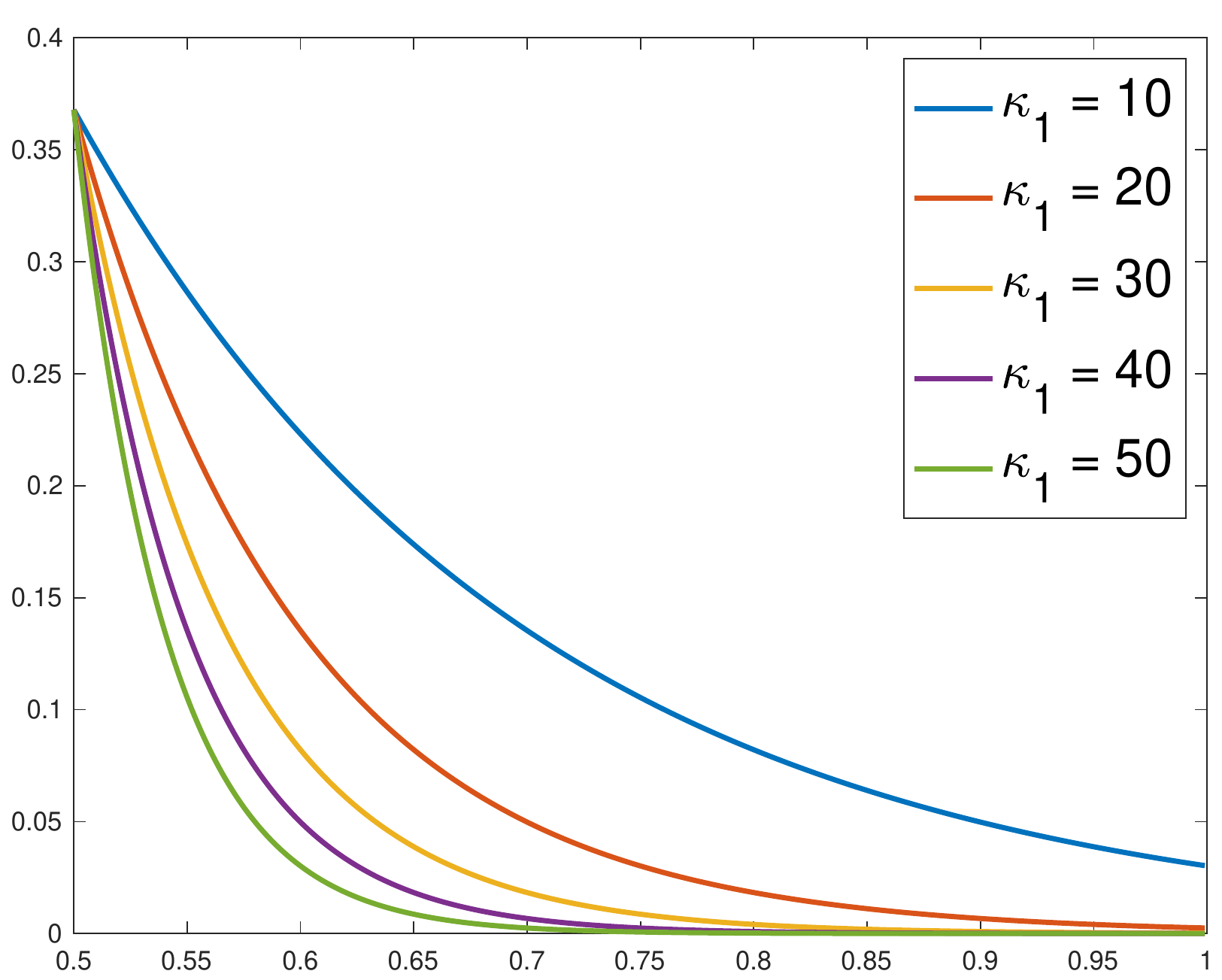}
     \label{fig:taufun1} 
\end{subfigure}
 \begin{subfigure}{.35\textwidth}
    \includegraphics[width=1.7in]{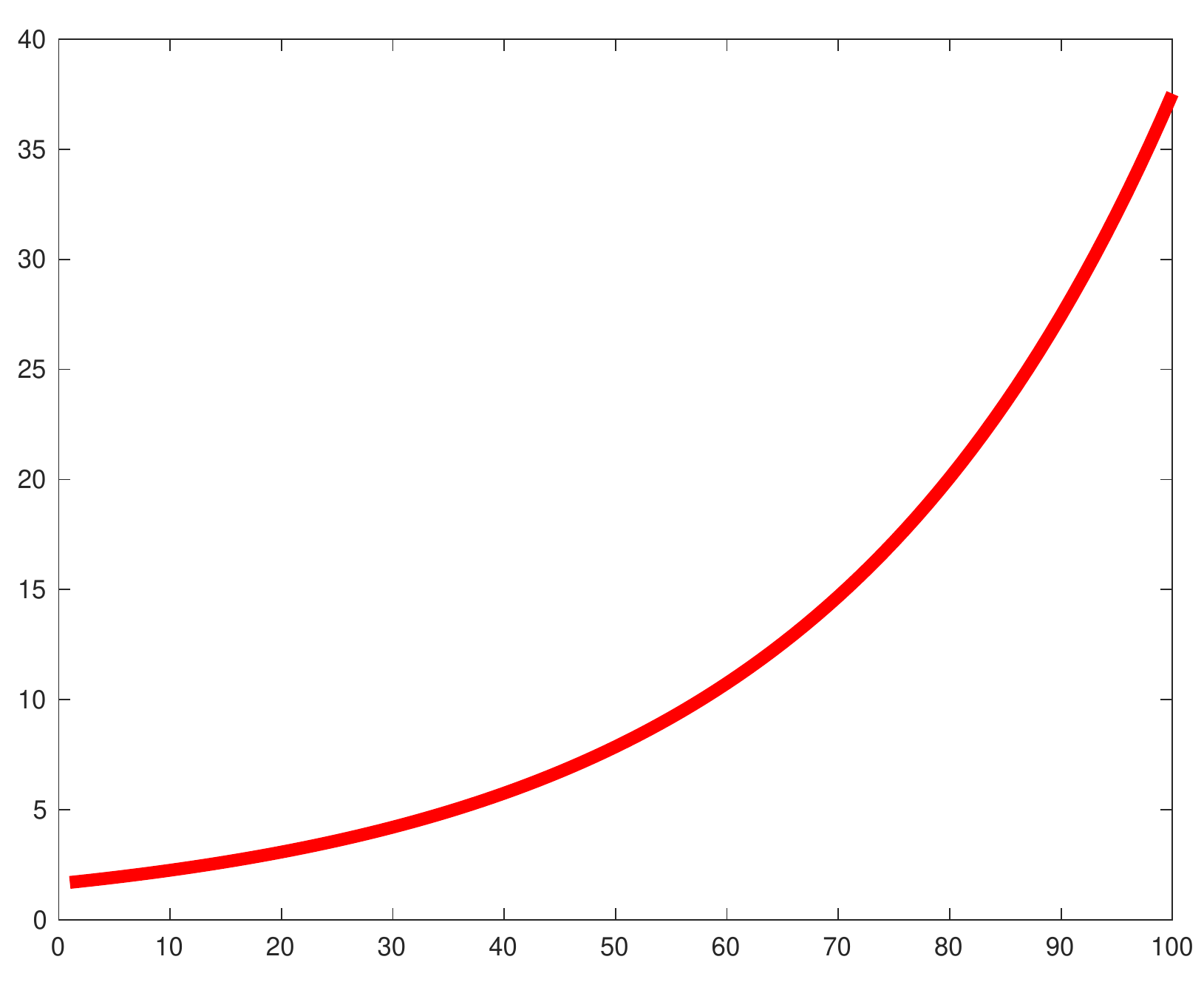}
     \label{fig:rho} 
\end{subfigure}
     \caption{Fix $\kappa_0 = -1$. Left: for various $\kappa_1$, $\tau(\theta)$ v.s. $\theta$. Right: $\varrho(\kappa_0,\kappa_1)$ v.s. $\kappa_1$.}
  \label{fig:taufun} 
\end{figure}




\item As for \hyperref[asp:polygonG2]{(G2)}, roughly speaking, \eqref{edge_cond_eq1} covers the case of a short edge, i.e., $h_e\rightarrow 0$,
while \eqref{edge_cond_eq2} allows a shrinking height supporting an edge, but the edge must be of a length $\mathcal{O}(h_K)$.
We illustrate a shrinking element in the left plot in Figure \ref{fig:submesh} where the edges $AB$ and $BC$ satisfy \eqref{edge_cond_eq1} and \eqref{edge_cond_eq2}, respectively.

\end{itemize}
\end{remark}





\subsection{Proof of Lemma \ref{lem_Pi_est}}
Thanks to the finite overlapping property assumed by \hyperref[asp:polygonG3]{(G3)} and the continuous extension given by Theorem \ref{thm_extension}, we only need to carry out the estimates locally on each element $K$ with a bound $\| \bfu \|_{E,s,\mathcal{B}_K}$.
Without loss of generality, we assume $K\subseteq \Omega^+$. For the first term in \eqref{lem_Pi_est_eq0}, we use Lemma \ref{lem_proj_approxi} and immediately obtain
\begin{equation}
\label{verify_lem_Pi_est_polygon_eq1}
\| \bfu - \Pi_K \bfu \|_{0,K} \le \| \bfu - \Pi_{\mathcal{B}_K} \mathcal{E}^+ \bfu \|_{0,K} \lesssim \|  \mathcal{E}^+ \bfu - \Pi_{\mathcal{B}_K} \mathcal{E}^+ \bfu \|_{0,\mathcal{B}_K}
\lesssim h^{\theta}_K |  \mathcal{E}^+ \bfu |_{\theta,\mathcal{B}_K}.
\end{equation}
Here, we point out that the extension $\mathcal{E}^+$ is needed as $\mathcal{B}_K$ may across the media interface.

As for the second term in \eqref{lem_Pi_est_eq0}, we consider the two cases for an edge $e$ satisfying \eqref{edge_cond_eq1} or \eqref{edge_cond_eq2} separately. The easy case is \eqref{edge_cond_eq1} for which we employ the following decomposition:
\begin{equation}
\label{verify_lem_Pi_est_polygon_eq2}
\| (I - \Pi_K)\bfu \cdot \bft \|_{0,e} \le \| (\bfu - \Pi_{\mathcal{B}_K} \mathcal{E}^+ \bfu ) \cdot \bft \|_{0,e} + \| (\Pi_{\mathcal{B}_K} \mathcal{E}^+ \bfu  - \Pi_K\bfu) \cdot \bft \|_{0,e}.
\end{equation}
As $e$ can be treated as one portion of an edge of a shape regular triangle in $\mathcal{B}_K$, the first term on the right-hand side above can be estimated by \cite[Lemma 7.2]{2017ErnGuermond}. We then focus on the second term and use \eqref{edge_cond_eq1} to write down
\begin{equation}
\begin{split}
\label{verify_lem_Pi_est_polygon_eq2_1}
&\| (\Pi_{\mathcal{B}_K} \mathcal{E}^+ \bfu  - \Pi_K\bfu) \cdot \bft \|_{0,e}  \le h_e^{1/2} | \Pi_{\mathcal{B}_K} \mathcal{E}^+ \bfu  - \Pi_K\bfu | = h_e^{1/2} |K|^{-1} \abs{ \int_K \Pi_{\mathcal{B}_K} \mathcal{E}^+ \bfu  - \bfu \dd \bfx } \\
 \le & h_e^{1/2}|K|^{-1/2}\| \Pi_{\mathcal{B}_K} \mathcal{E}^+ \bfu  - \bfu \|_{0,K} 
 \le 6 \sqrt{c_1/\pi} h^{s-1/2}_K | \mathcal{E}^+ \bfu |_{s,\mathcal{B}_K}.
\end{split}
\end{equation}

The difficulty is on the case \eqref{edge_cond_eq2}, as one cannot simply apply the argument in \eqref{verify_lem_Pi_est_polygon_eq2_1} due to the possibly-shrinking height, i.e., $h_Kh_e|K|^{-1}\le c_2 h_K/l_e \rightarrow \infty$. 
Here, we employ a completely different argument. 
Consider an edge $e_K$ that has the maximal length, 
and we know that $|e_K|\simeq h_K$. Let $l_K$ be the supporting height of ${e}_K$ toward $K$. 
We construct a square $\mathcal{S}_K\subseteq \mathcal{B}_K$ with also $\mathcal{S}_K\supseteq K$ such that it has one side denoted as $\bar{e}$ along $e_K$ with a length $\mathcal{O}(h_K)$. 
For simplicity of presentation, we may assume $\mathcal{S}_K$ is just a square with the side length $h_K$; otherwise we just slightly adjust its size.
See Figure \ref{fig:block} for the illustration of the geometry. 
For each $e\in\mathcal{E}_K$,  the trace inequality on the whole square $\mathcal{S}_K$ yields
 \begin{equation}
\label{verify_lem_Pi_est_polygon_eq3}
\| (I - \Pi_K)\bfu \cdot \bft \|_{0,e} \lesssim h^{-1/2}_K \| \mathcal{E}^+  \bfu - \Pi_{K} \mathcal{E}^+  \bfu  \|_{0,\mathcal{S}_K}. 
\end{equation}
We proceed to estimate the right-hand side in \eqref{verify_lem_Pi_est_polygon_eq3} and emphasize that it is certainly non-trivial as $K$ might be extremely shrinking to one edge of $\mathcal{S}_K$, and the low regularity makes the estimation even more difficult.

To estimate the right-hand side of \eqref{verify_lem_Pi_est_polygon_eq3}, we consider a local coordinate system $(x_1,x_2)$ with $x_1$ along with $\bar{e}$ and let $e(b_2)$ be the segment parallel to $\bar{e}$ with the $x_2$ coordinate $b_2\in (0,l_K)$.
Thus, $K$ is actually contained in a rectangle denoted as $\widetilde{\mathcal{S}}_K$ with a height $l_K$, as shown in Figure \ref{fig:block}.
In the following discussion, we write $\Pi_{e(b_2)}$ as $\Pi_{b_2}$ for simplicity.
We apply the triangular inequality to obtain
\begin{equation}
\begin{split}
\label{verify_lem_Pi_est_polygon_eq4}
\| \mathcal{E}^+  \bfu - \Pi_{K}  \mathcal{E}^+  \bfu  \|^2_{0,\mathcal{S}_K} & \lesssim \underbrace{ \| \mathcal{E}^+  \bfu - \Pi_{\mathcal{S}_K} \mathcal{E}^+  \bfu  \|^2_{0,\mathcal{S}_K} }_{(I)}  \\
& + \underbrace{ \| \Pi_{\mathcal{S}_K} \mathcal{E}^+  \bfu - \Pi_{b_2} \mathcal{E}^+  \bfu  \|^2_{0,\mathcal{S}_K} }_{(II)} 
+ \underbrace{ \|  \Pi_{b_2} \mathcal{E}^+  \bfu - \Pi_{K} \mathcal{E}^+  \bfu  \|^2_{0,\mathcal{S}_K} }_{(III)}.
\end{split}
\end{equation}
The estimate of $(I)$ immediately follows from Lemma \ref{lem_proj_approxi}. 
For $(II)$, with the projection property, the H\"older's inequality, the trace inequality and Lemma \ref{lem_proj_approxi}, we have
\begin{equation}
\begin{split}
\label{verify_lem_Pi_est_polygon_eq5}
(II) & 
= |\mathcal{S}_K| |e(b_2)|^{-2} \abs{ \int_{e(b_2)} \Pi_{\mathcal{S}_K} \mathcal{E}^+  \bfu - \mathcal{E}^+  \bfu  \dd \bfx  }^2 
 \le |\mathcal{S}_K| |e(b_2)|^{-1} \|  \Pi_{\mathcal{S}_K} \mathcal{E}^+  \bfu - \mathcal{E}^+  \bfu \|^2_{0,e(b_2)} \\
 &\lesssim \| \Pi_{\mathcal{S}_K} \mathcal{E}^+  \bfu - \mathcal{E}^+  \bfu \|^2_{0, \mathcal{S}_K} \lesssim h^{2s}_K | \mathcal{E}^+  \bfu |_{s,\mathcal{S}_K}.
\end{split}
\end{equation}
The most difficult one is $(III)$. 
With the similar derivation to \eqref{verify_lem_Pi_est_polygon_eq5}, we can rewrite $(III)$ into
\begin{equation}
\label{verify_lem_Pi_est_polygon_eq6}
(III) = |\mathcal{S}_K| |K|^{-2} \abs{ \int_{K} \Pi_{b_2} \mathcal{E}^+  \bfu - \mathcal{E}^+  \bfu  \dd \bfx  }^2 \le |\mathcal{S}_K| |K|^{-1} \|  \mathcal{E}^+  \bfu - \Pi_{b_2} \mathcal{E}^+  \bfu   \|^2_{0,\widetilde{S}_K}.
\end{equation}
Notice that $|\mathcal{S}_K| |K|^{-1}$ may blow up as $K$ may extremely shrink, and thus a more delicate analysis is needed.
We shall integrate $(III)$ in \eqref{verify_lem_Pi_est_polygon_eq6} with respect to $b_2$ from $0$ to $l_K$, 
and use the definition of $\Pi_{b_2}$ and the H\"older's inequality to obtain
\begin{equation*}
\begin{split}
\label{verify_lem_Pi_est_polygon_eq7}
\int_0^{l_K} (III) \dd b_2 & \le |\mathcal{S}_K| |K|^{-1} \int_0^{l_K} \int_0^{l_K} \int_{0}^{h_K} ( \mathcal{E}^+  \bfu(a_1,a_2) - \Pi_{b_2} \mathcal{E}^+  \bfu)^2 \dd a_1 \dd a_2 \dd b_2 \\
& \le |\mathcal{S}_K| |K|^{-1} \int_0^{l_K} \int_0^{l_K} \int_{0}^{h_K} h^{-2}_K \left( \int_0^{h_K} (\mathcal{E}^+  \bfu(a_1,a_2) - \mathcal{E}^+  \bfu(b_1,b_2)) \dd b_1 \right)^2 \dd a_1 \dd a_2 \dd b_2 \\
& \le |\mathcal{S}_K| |K|^{-1} h^{-1}_K \int_0^{l_K} \int_0^{l_K} \int_{0}^{h_K}  \int_0^{h_K} (\mathcal{E}^+  \bfu(a_1,a_2) - \mathcal{E}^+  \bfu(b_1,b_2))^2 \dd b_1 \dd a_1 \dd a_2 \dd b_2 \\
& \le |\mathcal{S}_K| |K|^{-1} h_K l^{2s}_K | \mathcal{E}^+ \bfu |^2_{s,\mathcal{S}_K} \le c^2_2 h^2_K l^{2s-1}_K | \mathcal{E}^+ \bfu |^2_{s,\mathcal{S}_K}  ,
\end{split}
\end{equation*}
where we have used Lemma \ref{lem_geometry_patch} and \eqref{edge_cond_eq2} in the last inequality. 
Now, dividing \eqref{verify_lem_Pi_est_polygon_eq7} by $l_K$, we arrive at  $(III) \le c^2_2 h^2_K l^{2s-2}_K  | \mathcal{E}^+ \bfu |^2_{s,\mathcal{S}_K}$.
Then, combining the estimates above we obtain
\begin{equation}
\label{verify_lem_Pi_est_polygon_eq8}
 \| \mathcal{E}^+  \bfu - \Pi_{K}  \mathcal{E}^+  \bfu  \|^2_{0,\mathcal{S}_K} \lesssim h^{2s}_K  | \mathcal{E}^+ \bfu |^2_{s,\mathcal{S}_K}   +  h^2_K l^{2s-2}_K | \mathcal{E}^+ \bfu |^2_{s,\mathcal{S}_K}   .
\end{equation}
Then, applying \eqref{rho_cont} in Assumption \hyperref[asp:polygonG1]{(G1)}, noticing $l_K\ge \rho_K$ and $\theta\in (1/2,1]$, and using \eqref{maxvc}, we have
\begin{equation}
\label{verify_lem_Pi_est_polygon_eq9}
h^2_K l^{2\theta-2}_K = h^{2\theta}_K (h_K/l_K)^{2(1-\theta)} \le h^{2\theta}_K (\tau(\theta))^{2(\theta-1)} \le h^{2\theta}_K (\varrho(\kappa_0,\kappa_1) )^2.
\end{equation}
Putting \eqref{verify_lem_Pi_est_polygon_eq9} into \eqref{verify_lem_Pi_est_polygon_eq8} and \eqref{verify_lem_Pi_est_polygon_eq3}, with the choice $\kappa_0 = 0.5$ and $\kappa_1 = 60$, we have
\begin{equation}
\label{verify_lem_Pi_est_polygon_eq10}
h^{1/2}_K \| (I - \Pi_K)\bfu \cdot \bft \|_{0,e} \lesssim  11 c_2 h^{s}_K  | \mathcal{E}^+ \bfu |_{s,\mathcal{B}_K},
\end{equation}
which yields the desired local estimate with \eqref{verify_lem_Pi_est_polygon_eq1}. 
The global one follows from the finite overlapping property imposed by Assumption \hyperref[asp:polygonG3]{(G3)}.

\begin{figure}[h]
\centering
    \includegraphics[width=2.5in]{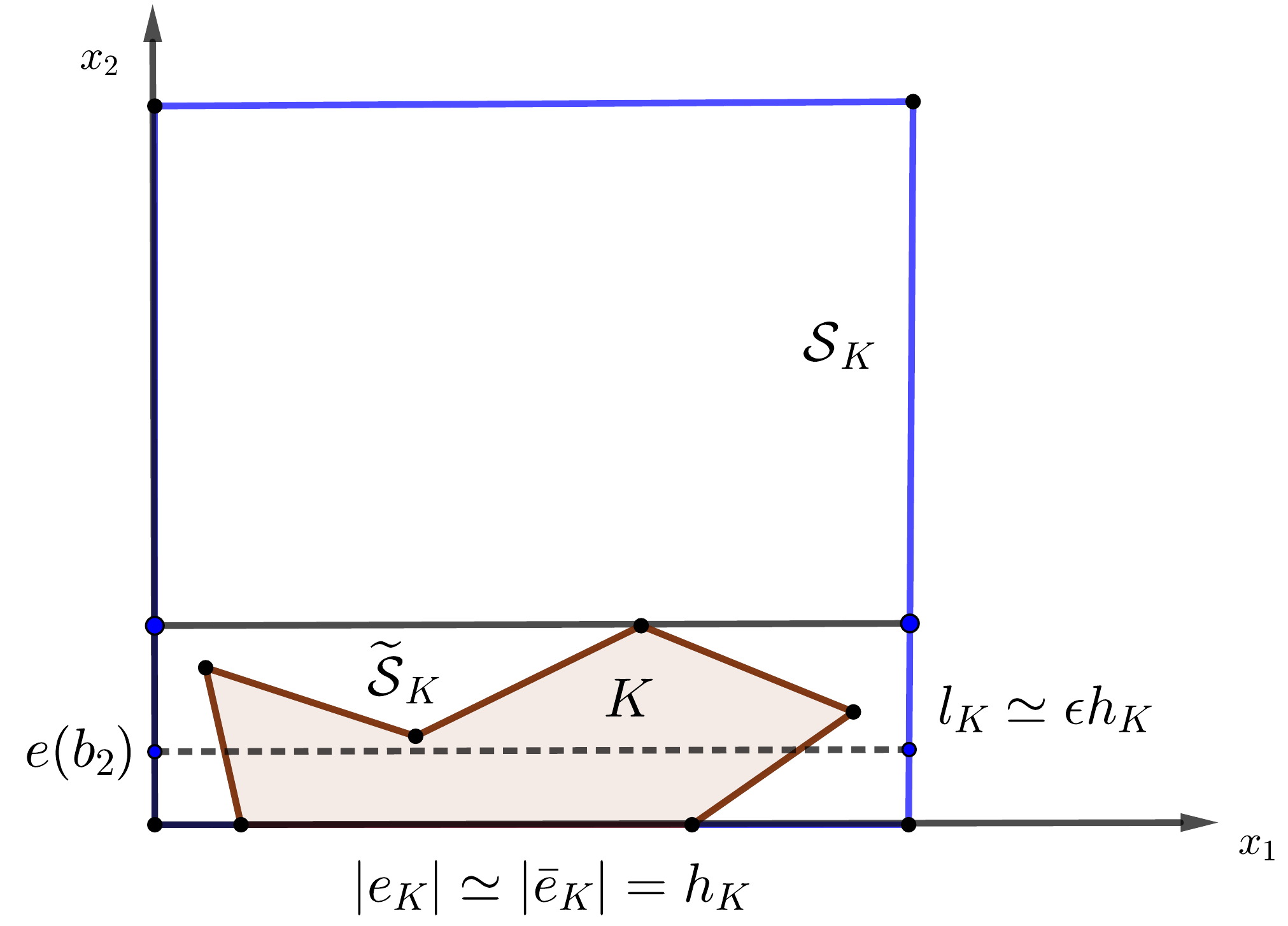}
    \caption{Illustration of an element $K$ and the auxiliary square $\mathcal{S}_K$. $K$ is contained in a rectangle $\widetilde{\mathcal{S}}_K$ with the height $l_K$.}
     \label{fig:block} 
\end{figure}



\subsection{Proof of Lemma \ref{lem_interp_est}}
Similarly, it is sufficient for us to carry out the estimation locally on each element.
The easiest one is \eqref{lem_interp_est_eq03} which immediately follows from the commutative diagram \eqref{DR_curl} with the approximation property of the projection given by Lemma \ref{lem_proj_approxi}

For \eqref{lem_interp_est_eq01}, the triangular inequality yields
\begin{equation}
\label{verify_lem_interp_est_eq2}
\| \bfxi_h  \|_{0,K} \le \| \bfu - \Pi_K \bfu \|_{0,K} + \| \Pi_K \bfu - I^e_h \bfu \|_{0,K}.
\end{equation}
The estimate of the first term on the right-hand-side of \eqref{verify_lem_interp_est_eq2} comes from Lemma \ref{lem_Pi_est}. 
For the second term, as $\Pi_K \bfu - I^e_h \bfu\in \bfV^e_h(K)$, $\forall K\in \mathcal{T}_h$, \eqref{lem_Poincare_eq0} in Lemma \ref{lem_L2stab} with the triangular inequality implies
\begin{equation}
\begin{split}
\label{verify_lem_interp_est_eq3}
 \| \Pi_K \bfu - I^e_h \bfu \|_{0,K} 
  \lesssim h_K \| \rot I^e_h \bfu \|_{0,K} +  h^{1/2}_K \|  (\Pi_K \bfu - \bfu)\cdot\bft \|_{0,\partial K} +  h^{1/2}_K \|  ( \bfu - I^e_h \bfu)\cdot\bft \|_{0,\partial K}.
 \end{split}
\end{equation}
Here, the bound of the first term on the right-hand-side above simply follows from \eqref{lem_interp_est_eq03}, 
while the estimate of the second term follows from Lemma \ref{lem_Pi_est}. Now, let us concentrate on third term in the right-hand side of \eqref{verify_lem_interp_est_eq3}. 
Given each edge $e\in\mathcal{E}_K$, suppose $e\subseteq\Omega^+$ without loss of generality, and let $\tilde{e}\subseteq \mathcal{B}_K$ be an extension of $e$ such that the length is $\mathcal{O}(h_K)$.
Then, the projection property, the trace inequality of \cite[Lemma 7.2]{2017ErnGuermond} and Lemma \ref{lem_proj_approxi} yield
\begin{equation}
\begin{split}
\label{verify_lem_interp_est_eq12}
\| ( \bfu - I^e_h \bfu)\cdot\bft \|_{0,e} & = \| \bfu\cdot\bft - \Pi_e (\bfu\cdot\bft) \|_{0,e} \le \| (\mathcal{E}^+\bfu)\cdot\bft - (\Pi_{\mathcal{B}_K} \mathcal{E}^+\bfu)\cdot\bft \|_{0,\tilde{e}} \\
& \lesssim h^{-1/2}_K \| \mathcal{E}^+\bfu - \Pi_{\mathcal{B}_K} \mathcal{E}^+\bfu \|_{0,\mathcal{B}_K} + h^{s-1/2}_K | \mathcal{E}^+\bfu |_{s,\mathcal{B}_K} \lesssim h^{s-1/2}_K | \mathcal{E}^+\bfu |_{s,\mathcal{B}_K}.
\end{split}
\end{equation}
Thus, \eqref{lem_interp_est_eq01} is concluded by putting these estimates into \eqref{verify_lem_interp_est_eq3} and using the block condition \hyperref[asp:polygonG3]{(G3)} in Assumption \ref{assump_polygon}.

Next, for \eqref{lem_interp_est_eq02}, by the triangular inequality and applying \eqref{lem_Poincare_eq0} to $\Pi_K\bfxi_h$, we have
\begin{equation}
\begin{split}
\label{verify_lem_interp_est_eq4}
\|  \bfxi_h \|^2_{b} 
 \lesssim h \sum_{K\in\mathcal{T}_i} ( \| \Pi_K \bfxi_h\cdot\bft \|^2_{0,\partial K} + \| \bfxi_h\cdot\bft \|^2_{0,\partial K} ).
\end{split}
\end{equation}
Here, the second term on the right-hand side, $\| \bfxi_h\cdot\bft \|^2_{0,\partial K}$, is the same as \eqref{verify_lem_interp_est_eq12}.
Let us focus on $\| \Pi_K \bfxi_h\cdot\bft \|_{0,\partial K}$ below.

We still separate our discussion into the two situations given by \eqref{edge_cond}. 
Given an edge $e\in\mathcal{E}_K$, let us first consider \eqref{edge_cond_eq1}. 
As $\Pi_K\bfxi_h$ is a constant vector on $K$, we have 
\begin{equation}
\label{verify_lem_interp_est_polygon_eq1}
 \| \Pi_K \bfxi_h \cdot\bft \|_{0,e} =  h_e^{1/2}/|K|^{-1} \abs{ \int_K \bfxi_h \cdot\bft \dd \bfx } \le \sqrt{c_1} h^{-1/2}_K \| \bfxi_h \|_{0,K},
\end{equation}
of which the estimate is given by \eqref{lem_interp_est_eq01} established above.


Now, we suppose that $e$ satisfies \eqref{edge_cond_eq2}. 
As the element $K$ may extremely shrink, the argument in \eqref{verify_lem_interp_est_eq4} is not applicable anymore.
Without loss of generality, we can let $e$ be along with the $x_1$ axis, and thus $\bft_e = [1,0]^\top$. 
Construct $p^e_h = x_2\in \mathbb{P}_1(K)$ such that $\brot p^e_h = \bft_e$.
Then, with the projection property and integration by parts, we have
\begin{equation}
\begin{split}
\label{verify_lem_interp_est_polygon_eq5}
\| \Pi_K\bfxi_h\cdot\bft_e \|_{0,e}& = h^{1/2}_e |K|^{-1} \abs{ \int_K \bfxi_h\cdot \brot p^e_h \dd \bfx  } \\
& \le \underbrace{ h^{1/2}_e |K|^{-1} \| \rot \bfxi_h \|_{0,K} \| p^e_h \|_{0,K} }_{(I)} +  \underbrace{ h^{1/2}_e |K|^{-1} \| \bfxi_h \cdot \bft \|_{0,\partial K}  \|  p^e_h \|_{0,\partial K} }_{(II)}.
\end{split}
\end{equation}
Noticing $ \| p^e_h \|_{\infty,K} \le l_e$, we apply \eqref{edge_cond_eq2} together with \eqref{verify_lem_interp_est_eq12} and \eqref{lem_interp_est_eq03} to obtain
\begin{subequations}
\label{verify_lem_interp_est_polygon_eq6}
 \begin{align}
(I) & \le h^{1/2}_e |K|^{-1/2} l_e  \| \rot \bfxi_h \|_{0,K} \le \sqrt{c_2} l^{1/2}_e \| \rot \bfu \|_{0,K} \lesssim \sqrt{c_2}  h^{1/2}_K \| \rot \bfu \|_{0,K}, \\
(II) & \le h^{1/2}_e |K|^{-1} |\partial K|^{1/2} l_e \| \bfxi_h \cdot \bft \|_{0,\partial K}  \lesssim c_2 h_K^{1/2} h^{-1/2}_e |  \mathcal{E}^+  \bfu |_{s,\mathcal{B}_K} \lesssim c^{3/2}_2  h^{s-1/2}_K |  \mathcal{E}^+  \bfu |_{s,\mathcal{B}_K}.
\end{align}
\end{subequations}
At last, the proof is finished by putting \eqref{verify_lem_interp_est_polygon_eq6} into \eqref{verify_lem_interp_est_polygon_eq5}.

\section{Numerical Examples}
\begin{figure}[h]
    \centering
\begin{subfigure}{.48\textwidth}
    \centering
    \includegraphics[width=1.5in]{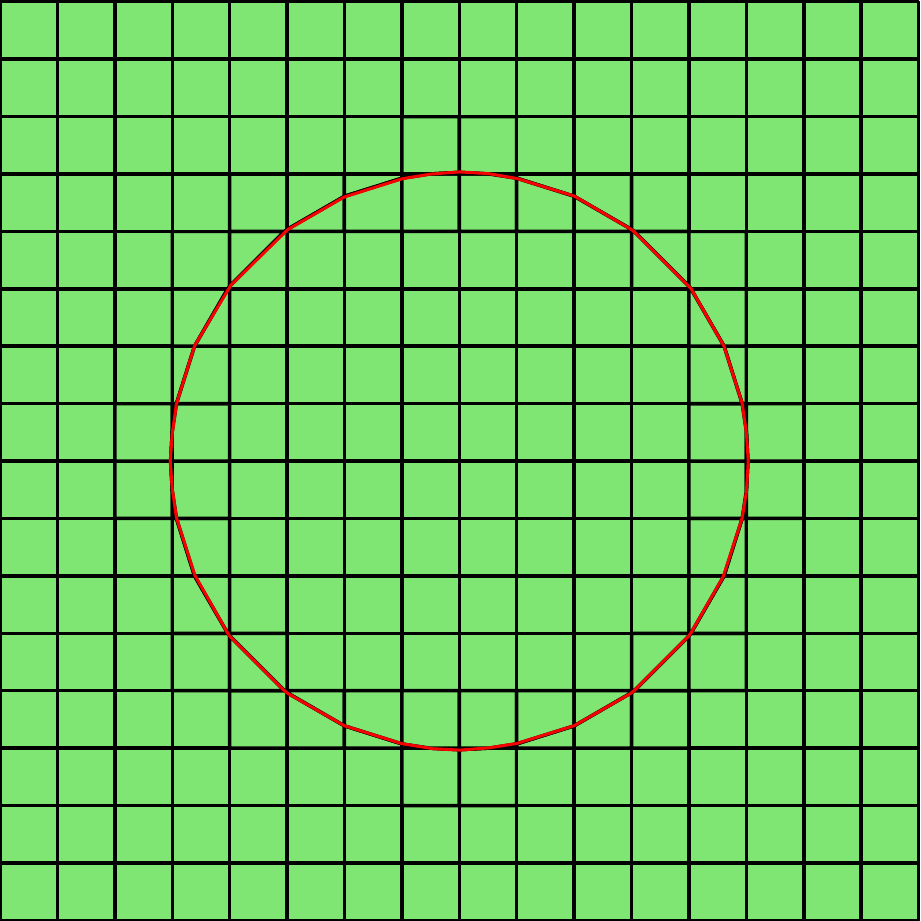}
\end{subfigure}
 \begin{subfigure}{.48\textwidth}
    \centering
    \includegraphics[width=1.5in]{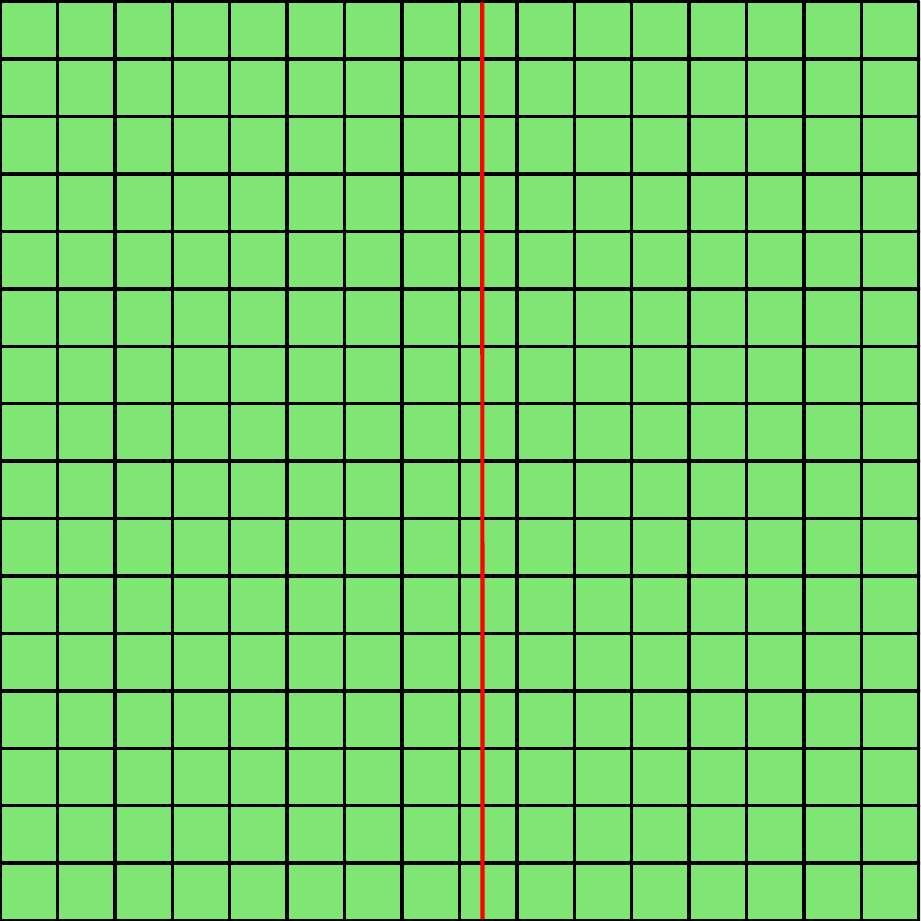}
\end{subfigure}
\caption{Left: Example \ref{subsec:examp1} with
    mesh size $h = 1/16$. Right: Example \ref{subsec:examp2} with
    mesh size $h = 1/16$.}
    \label{fig:example12}
\end{figure}
In this section, we shall numerically demonstrate the effectiveness of the proposed method. 
The first two examples aim to verify the convergence and robustness of the method, 
with the interface being a circle and a straight line, respectively. 
In the third example, we consider a more complex interface shape with a geometric singularity. 
In the final example, we apply the method to the case of extremely thin layers. 
The last two examples all feature relatively large wave numbers.
All meshes are generated by cutting the background Cartesian mesh directly by the interface.
The code is implemented by using the FEALPy package \cite{fealpy}.

\subsection{Convergence order}
\label{subsec:examp1}
    Set the domain and the interface geometry:
    $$
    \Omega = (-1, 1)\times(-1, 1), \quad \Omega^- = \{(x, y) : x^2+y^2 <
    \frac{\pi}{5}\},\quad \Omega^+ = \Omega - \Omega^-,
    $$
    which is illustrated by the left plot in Figure \ref{fig:example12}.
    Consider the parameters $\alpha|_{\Omega^-} = \beta|_{\Omega^-} = 1, \alpha|_{\Omega^+} =
    \beta|_{\Omega^+} = 10$, and the exact solution:
    $$
    {\bf u} = \left\{
        \begin{aligned}
        & - k_0(r_0^2 - x^2 - y^2)
        \begin{pmatrix}
        y\\
        x
        \end{pmatrix} \quad& \text{in } \Omega^-, \\
        &-\frac{1}{10}
        k_1(r_0^2 - x^2 - y^2)(r_1^2 - x^2 - y^2)
        \begin{pmatrix}
        y\\
        x
        \end{pmatrix}&\text{ in } \Omega^+, \\
        \end{aligned}
    \right.
    $$
    where $r_0 = \frac{\pi}{5}, r_1 = 1, k_1 = 20, k_0 = k_1(r_1^2-r_0^2).$
The numerical results are presented in Table \ref{tab:exm0} where we can see that $\|\bfu - \Pi_h \bfu_h\|_0 = o(h), \|\rot \bfu - \rot
\Pi_h\bfu_h\|_0 = o(h)$ which is clearly optimal.
\begin{table}[!htp] 
\centering
\caption{Errors $\|\bfu - \Pi_h \bfu_h\|_0$ and $\|\rot \bfu - \rot \bfu_h\|_0$ 
of Example \ref{subsec:examp1}.}
\label{tab:exm0}
\begin{tabular}[c]{|c|c|c|c|c|c|c|c|}\hline
$h$ & $1/8$ & $1/16$ & $1/32$ & $1/64$ & $1/128$
\\\hline
$|| {\bf u} - \Pi_h{\bf u}_h||_0$ & 4.7980e-01 & 2.1925e-01 & 1.0073e-01 & 4.7611e-02 & 2.3031e-02
\\\hline
Order & -- & 1.13 & 1.12 & 1.08 & 1.05
\\\hline
$||\rot {\bf u} - \rot {\bf u}_h||_0$ & 1.1556e+00 & 5.7222e-01 & 2.7268e-01 & 1.3669e-01 & 6.8468e-02
\\\hline
Order & -- & 1.01 & 1.07 & 1.   & 1.
\\\hline
\end{tabular}
\end{table}


\subsection{Robustness with respect to the element shape}
\label{subsec:examp2}

   Still consider the domain 
    $\Omega=(-1, 1)\times(-1, 1)$ but with a linear interface 
    $\Gamma = \{(x, y): x=\varepsilon, -1<y<1\}$, see the right plot in Figure \ref{fig:example12} for illustration.
    The exact solution and the parameters are given by
    $$
    \bfu = 
    \begin{pmatrix}
    |x-\varepsilon|^s+\cos(x+y)\\
    \sin(x+y)
    \end{pmatrix},\ \alpha = 1, \ 
    \beta|_{\Omega^-} = 1, \beta|_{\Omega^+} = 2,
    $$
    where $\Omega^-=\{(x, y) : x>\varepsilon\}, \Omega^+=\Omega-\Omega^-$.
In this case, the true solution $\bfu$ does not have a $H^1$ regularity when $s < 0.5$.
We let $\varepsilon = 1e-7$ and $h = 2^{-k}, k = 3, 4, 5, 6, 7$, such that very thin elements will be generated near the interface. 
The numerical results for different $s$ are 
presented in Tables \ref{tab:error1}-\ref{tab:error3} which still show the optimal convergence rate in terms of the regularity order, 
i.e., $\|\bfu - \Pi_h \bfu_h\|_0 = o(h^{s+0.5})$. 
It also shows that the VEM is highly robust to the very irregular and anisotropic element shapes.


\begin{table}[!h]
\centering
\caption{Errors $\|\bfu - \Pi_h \bfu_h\|_0$ and $\|\rot \bfu - \rot \bfu_h\|_0$ 
of Example \ref{subsec:examp2} with $s=0.2$.}
\label{tab:error1}
\begin{tabular}[c]{|c|c|c|c|c|c|c|}\hline
$h$ & $1/8$ & $1/16$ & $1/32$ & $1/64$ & $1/128$
\\\hline
$|| \bfu - \Pi_h\bfu_h||_0$ & 4.4285e-01 & 2.7290e-01 & 1.6819e-01 & 1.0362e-01 & 6.3819e-02
\\\hline
Order & -- & 0.7  & 0.7  & 0.7  & 0.7
\\\hline
$||\rot {\bf u} - \rot {\bf u}_h||_0$ & 2.1385e-01 & 1.0371e-01 & 5.1362e-02 & 2.5594e-02 & 1.2779e-02
\\\hline
Order & -- & 1.04 & 1.01 & 1.   & 1.
\\\hline
\end{tabular}
\end{table}

\begin{table}[!h]
\centering
\caption{Errors $\|\bfu - \Pi_h \bfu_h\|_0$ and $\|\rot \bfu - \rot \bfu_h\|_0$ 
of Example \ref{subsec:examp2} with $s=-0.1$.}
\label{tab:error2}
\begin{tabular}[c]{|c|c|c|c|c|c|c|}\hline
$h$ & $1/8$ & $1/16$ & $1/32$ & $1/64$ & $1/128$
\\\hline
$|| \bfu - \Pi_h\bfu_h||_0$ & 1.2115e+00 & 9.1955e-01 & 6.9733e-01 & 5.2857e-01 & 4.0059e-01
\\\hline
Order & -- & 0.4 & 0.4 & 0.4 & 0.4
\\\hline
$||\rot {\bf u} - \rot {\bf u}_h||_0$ & 2.4553e-01 & 1.1701e-01 & 5.7737e-02 & 2.8908e-02 & 1.4583e-02
\\\hline
Order & -- & 1.07 & 1.02 & 1.   & 0.99
\\\hline
\end{tabular}
\end{table}

\begin{table}[!h]
\centering
\caption{Errors $\|\bfu - \Pi_h \bfu_h\|_0$ and $\|\rot \bfu - \rot \bfu_h\|_0$ 
of Example \ref{subsec:examp2} with $s=-0.4$.}
\label{tab:error3}
\begin{tabular}[c]{|c|c|c|c|c|c|c|}\hline
$h$ & $1/8$ & $1/16$ & $1/32$ & $1/64$ & $1/128$
\\\hline
$|| \bfu - \Pi_h\bfu_h||_0$ & 3.5682e+00 & 3.3247e+00 & 3.1026e+00 & 2.8985e+00 & 2.7102e+00
\\\hline
Order & -- & 0.1  & 0.1  & 0.1  & 0.1
\\\hline
$||\rot {\bf u} - \rot {\bf u}_h||_0$ & 4.9427e-01 & 2.8425e-01 & 1.7532e-01 & 1.1181e-01 & 7.2487e-02
\\\hline
Order & -- & 0.8  & 0.7  & 0.65 & 0.63
\\\hline
\end{tabular}
\end{table}


\subsection{Geometric singularity}
\label{subsec:geo_sing}
In this example, we consider an interface shape formed by two circles containing certain geometric singularity:
$\Omega=(0, 4)\times(0, 1)$, $\Omega^+ = \Omega-\Omega^-$, and 
\begin{figure}[h]
    \centering
    \includegraphics[width=4.4in]{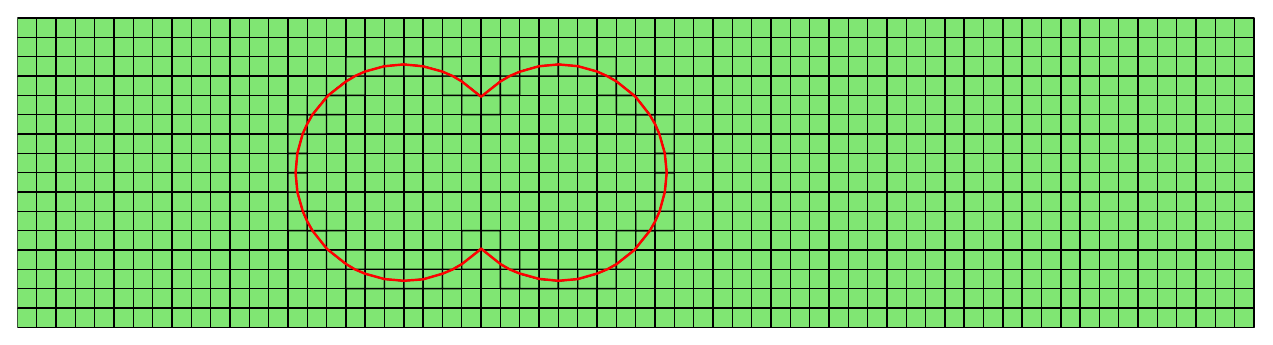}
    \caption{Geometry for Example \ref{subsec:geo_sing} with the mesh size $h=1/16$.}
     \label{fig:interfacedoublecircle} 
\end{figure}
    $$
    \Omega^- = \{(x, y) : (x-1.75)^2+y^2 < r^2  \} \cup \{(x, y) :
    (x-1.25)^2+y^2 < r^2  \},
    $$
    where $r = 0.35$, see Figure \ref{fig:interfacedoublecircle} for illustration.
    Set the source term, parameters and the boundary condition to be
    $$
    \alpha = 1,\quad \beta = \omega^2\left(\epsilon + {\rm i} \frac{\sigma}{\omega}\right), \quad {\bf f} = -{\rm i}\omega\begin{pmatrix}
0\\
1
\end{pmatrix} \exp\left(-\frac{(x-3)^2}{\epsilon^2}\right),\quad g = 0.
    $$
    where ${\rm i} = \sqrt{-1}, \sigma|_{\Omega^-}=1, 
    \sigma|_{\Omega^+}=0.1$.

As it is very difficult to construct functions in an explicit format that can exactly satisfy the jump conditions,
we will use the solution computed on the mesh with the size $h = 1/2^8$ as the reference solution, denoted as $\bfu^8_h$.
The numerical results for different $\omega$ and $\epsilon$ are 
presented in Tables \ref{tab:exm3-0}-\ref{tab:exm3-2}, which clearly indicate the optimal convergence.
Meanwhile, we also compare the numerical solutions of the VEM and classic FEM with
the first kind N\'ed\'elec element for the mesh size $h = 1/2^8$. The numerical results are shown
in Figure \ref{fig:femvsvemdoublecircle}, where we can clearly observe that their behavior are quite comparable.

\begin{table}[!h]
\centering
\caption{The solutions errors of Example \ref{subsec:geo_sing} with $\epsilon = 0.5$, $\omega=5$}
\label{tab:exm3-0}
\begin{tabular}[c]{|c|c|c|c|c|c|}\hline
$k$ & $3$ & $4$ & $5$ & $6$ & $7$
\\\hline
$||{\bf u}_h^8 - {\bf u}_h^k||_0$ & 6.2347e-02 & 2.6444e-02 & 1.2342e-02 & 5.9152e-03 & 2.6343e-03
\\\hline
Order & -- & 1.24 & 1.1  & 1.06 & 1.17
\\\hline
$||{\rm rot} {\bf u}_h^8 - {\rm rot} {\bf u}_h^k||_0$  & 2.2230e-01 & 9.5086e-02 & 4.4559e-02 & 2.1380e-02 & 9.5216e-03
\\\hline
Order & -- & 1.23 & 1.09 & 1.06 & 1.17
\\\hline
\end{tabular}
\label{tab:eps05}
\end{table}


\begin{table}[!htp]
\centering
\caption{The solution errors of Example \ref{subsec:geo_sing} with $\epsilon = 0.01, \omega=100$}
\label{tab:exm3-2}
\begin{tabular}[c]{|c|c|c|c|c|c|c|c|}\hline
$k$ & $3$ & $4$ & $5$ & $6$ & $7$
\\\hline
$|| \Pi_h\bfu_h^8 - \Pi_h\bfu_h^k||_0$ & 7.9611e-02 & 3.6193e-02 & 1.2285e-02 & 5.0727e-03 & 2.1258e-03
\\\hline
Order & -- & 1.14 & 1.56 & 1.28 & 1.25
\\\hline
$||\rot \bfu_h^8 - \rot \bfu_h^k||_0$ & 8.5545e-01 & 3.9451e-01 & 1.5309e-01 & 7.1350e-02 & 3.0665e-02
\\\hline
Order & -- &  1.12 &  1.37 &  1.1  &  1.22
\\\hline
\end{tabular}
\end{table}

\begin{figure}[!h]
    \centering
    \includegraphics[width=5in]{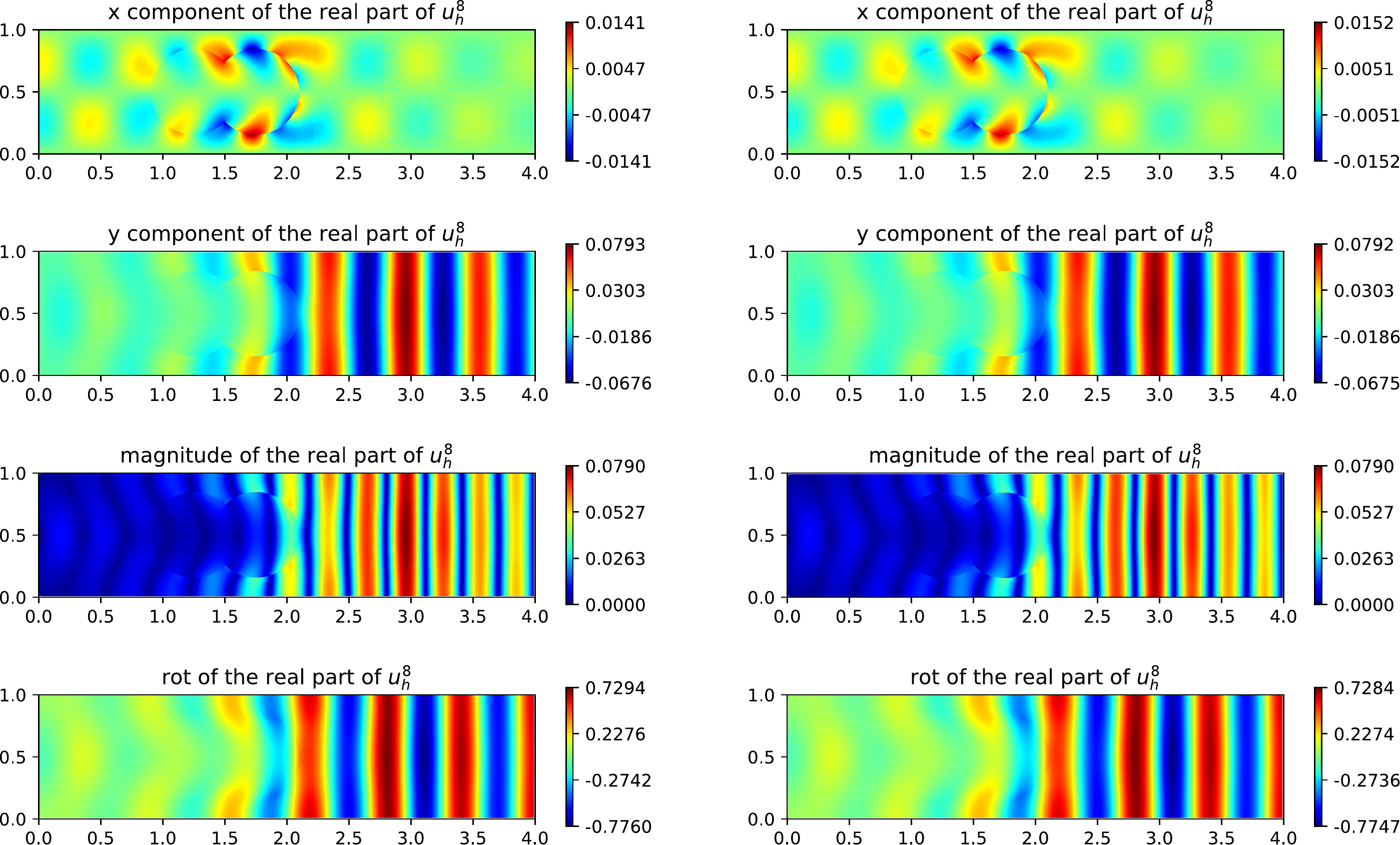}
    \caption{Example \ref{subsec:geo_sing}: comparison of the FEM solution with the
            first kind N\'ed\'elec element (left) and the VEM solution (right) when the mesh size $h = 1/2^8$.}
     \label{fig:femvsvemdoublecircle} 
\end{figure}

\subsection{Thin layers}
\label{subsec:thinlayer}
Finally, we consider the scenario where electromagnetic waves propagate through extremely thin layers, 
specifically two or five layers as shown in Figure \ref{fig:twolayerspolymesh}.
This situation requires very fine meshes to accurately fit the thin layers, as
illustrated in Figure \ref{fig:twolayerstrimesh}, which poses challenges for mesh generation.
    Let $\Omega=(0, 4)\times(0, 1)$ and
    $$
    \Omega^- = \left\{
        \begin{aligned}
            &\{(x, y) : x \in (0, 4), y \in (0.24, 0.26)\cup(0.74, 0.76)\}\quad 
            & \text{\rm for the case of two layers,}\\
            &\begin{multlined}
            \{(x, y) : x \in (0, 4),
            y \in (0.09, 0.11)\cup(0.24, 0.26)\\
            \cup(0.49, 0.51)\cup(0.74,
            0.76)\cup(0.89, 0.91)\}
            \end{multlined}
            \quad 
            & \text{\rm for the case of five layers.}\\
        \end{aligned}
    \right.
    $$
Set the parameters, the source term and the boundary condition as
    $$
    \alpha = 1,\quad \beta = \omega^2\left(\epsilon + {\rm i} \frac{\sigma}{\omega}\right), \quad {\bf f} = -{\rm i}\omega\begin{pmatrix}
0\\
1
\end{pmatrix} \exp\left(-\frac{(x-3)^2}{\epsilon^2}\right),\quad g = 0.
    $$
    where ${\rm i} = \sqrt{-1}, \omega=100, \sigma|_{\Omega_0}=1, \sigma|_{\Omega_1}=0.1$.

Let us first present the results for the case of two layers. 
Similar to Example \ref{subsec:geo_sing}, we still use the solution computed on the fine mesh with the size $h = 1/2^8$ as the reference solution, denoted as $\bfu^8_h$.
The numerical results are presented in Table \ref{tab:exm4} which clearly show the optimal convergence.
In addition, we also compare the VEM solution and the classical FEM solutions with the first kind N\'ed\'elec element method for the mesh size $h = 1/2^8$. 
The results are shown in Figure \ref{fig:femvsvemdoublelayer} and \ref{fig:femvsvemfivelayer}, respectively for the cases of two layers and five layers.
Again, their behavior are quite comparable.

\begin{figure}[!h]
    \centering
    \includegraphics[width=4.4in]{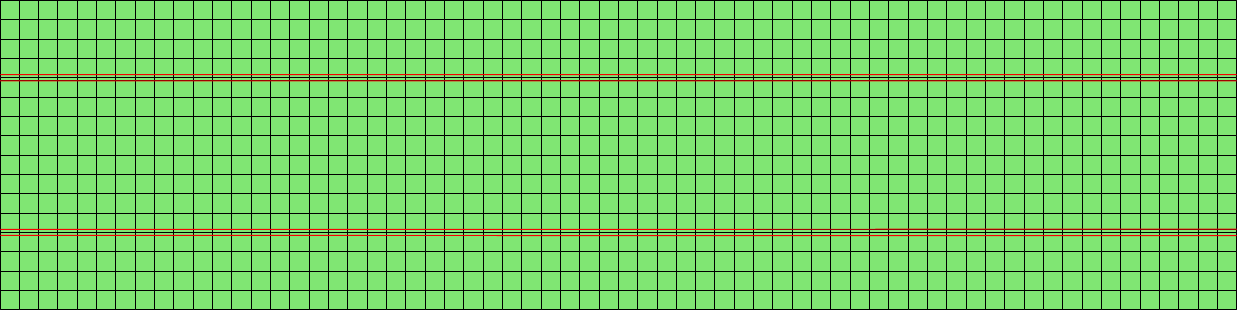}
    \caption{The mesh for Example \ref{subsec:thinlayer} with two layers at mesh size
    $h=1/16$.}
     \label{fig:twolayerspolymesh} 
\end{figure}

\begin{figure}[!h]
    \centering
    \includegraphics[width=4.4in]{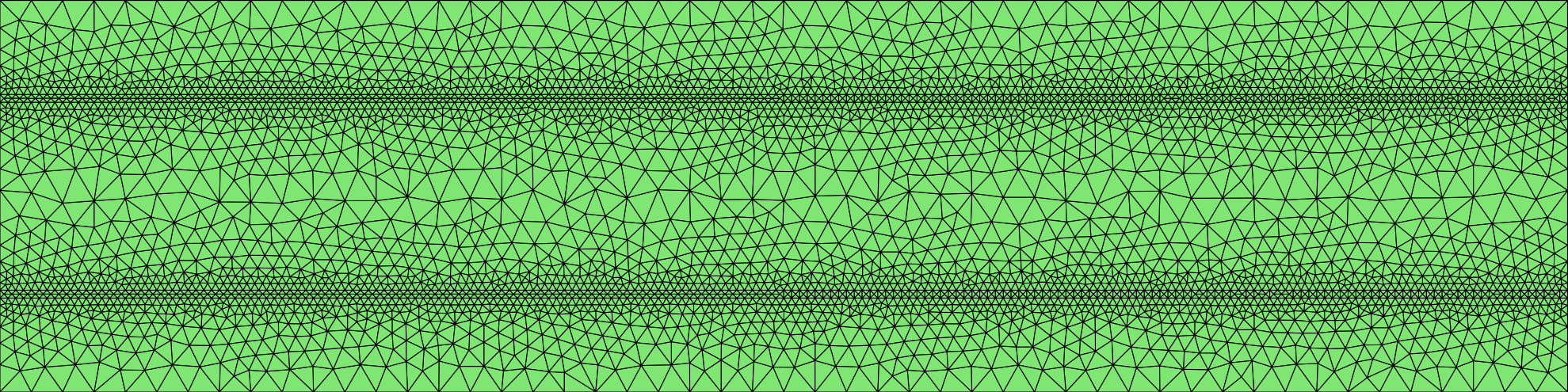}
    \caption{The triangle mesh for Example \ref{subsec:thinlayer} with two
        layers.}
     \label{fig:twolayerstrimesh} 
\end{figure}

\begin{table}[!h]
\centering
\caption{The solution errors of Example \ref{subsec:thinlayer}}
\label{tab:exm4}
\begin{tabular}[c]{|c|c|c|c|c|c|c|}\hline
$k$ & $3$ & $4$ & $5$ & $6$ & $7$
\\\hline
$||\Pi_h\bfu_h^8 - \Pi_h \bfu_h^k||_0$ & 8.7974e-02 & 4.3586e-02 & 1.4656e-02 & 5.6111e-03 & 2.2800e-03
\\\hline
Order & -- & 1.01 & 1.57 & 1.39 & 1.3 
\\\hline
$||\rot \bfu_h^8 - \rot \bfu_h^k||_0$ & 9.3078e-01 & 4.6241e-01 & 1.7258e-01 & 7.4942e-02 & 3.1537e-02
\\\hline
Order & -- & 1.01 & 1.42 & 1.2  & 1.25
\\\hline
\end{tabular}
\end{table}

\begin{figure}[!h]
    \centering
    \includegraphics[width=5.0in]{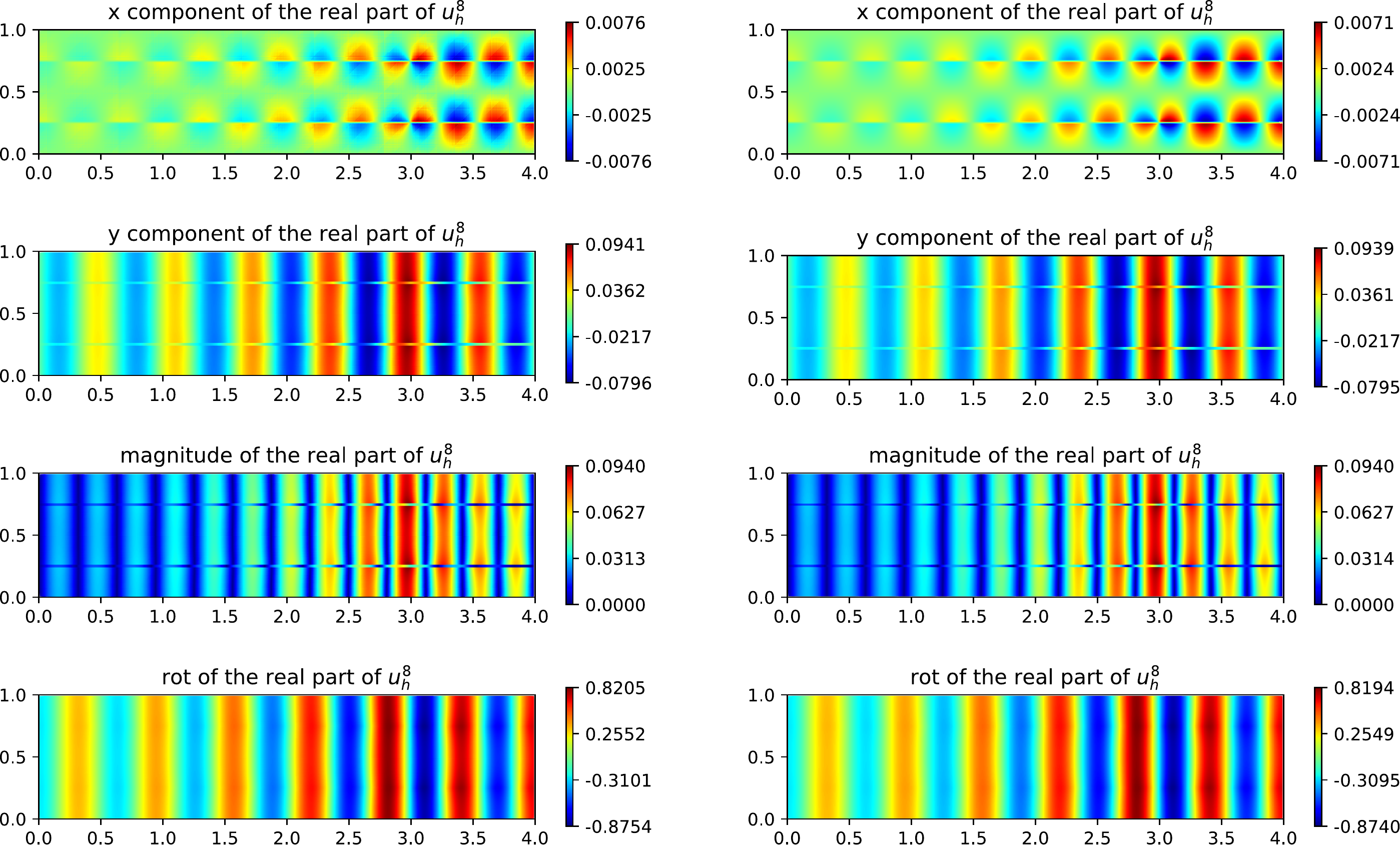}
    \caption{Example \ref{subsec:thinlayer} with two layers: comparison of the FEM solution with the
    first kind N\'ed\'elec element (left) and the VEM solution (right) when the mesh size $h = 1/2^8$.}
    \label{fig:femvsvemdoublelayer} 
\end{figure}

\begin{figure}[!h]
    \centering
    \includegraphics[width=5.0in]{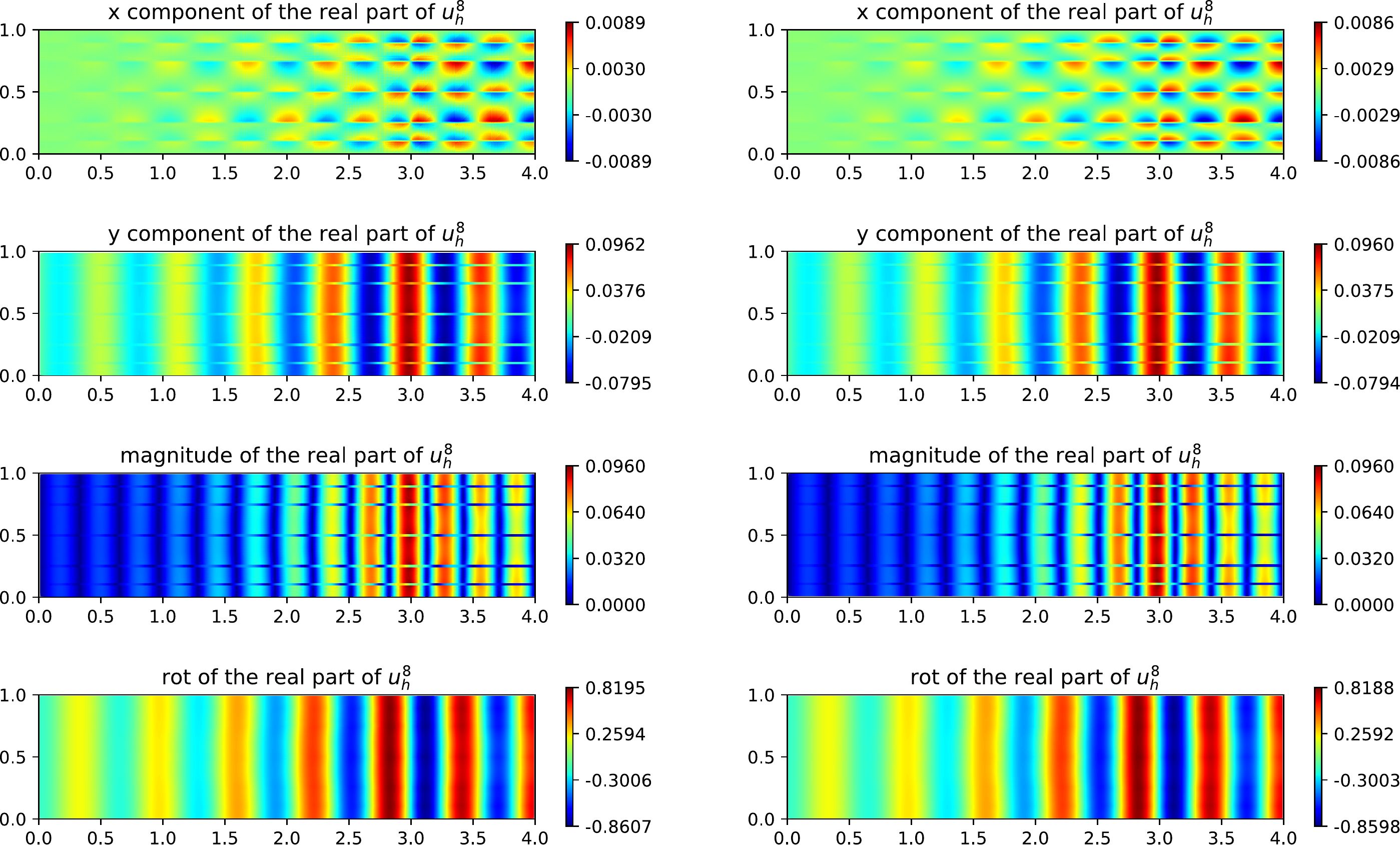}
    \caption{Example \ref{subsec:thinlayer} with five layers: comparison of the FEM solution with the
    first kind N\'ed\'elec element (left) and the VEM solution (right) when the mesh size $h = 1/2^8$.}
     \label{fig:femvsvemfivelayer} 
\end{figure}

 
\appendix

\section{Proof of Lemma \ref{lem_embd_1} for complex $\beta$}
\label{append_embed}
 Denote $\beta = \beta_R + i \beta_I$ and assume $\beta_R > 0$.
 Given $\bfu = \bfu_I + i\bfu_R \in \bfH_0(\rot;\Omega)$, we let $\bfu_I$ and $\bfu_R$ be column vectors. Define a matrix
 \begin{equation}
\label{lem_embd_1_eq1}
\bfB = \frac{1}{\beta^2_R + \beta^2_I} \left[ \begin{array}{cc} \beta_R & \beta_I \\ -\beta_I & \beta_R \end{array}\right] .
\end{equation}
As $\bfB$ is positive definite, we consider a system for a column vector function $\bfphi = [\phi_R,\phi_I]^T$ satisfying
\begin{equation}
\label{lem_embd_1_eq1_1}
\rot(\bfB \, \brot \bfphi ) = - \ddiv\left( \bfB \nabla \bfphi \right) =  \rot( [\bfu_R ,\bfu_I ]^T ), ~~~~ [\bfphi]_{\Gamma} = \mathbf{0}, ~~~~ [ (\bfB \nabla \bfphi)\bfn ]_{\Gamma} = 0, 
\end{equation}
subject to $\nabla\bfphi\cdot\bfn = \mathbf{ 0}$  on $\partial \Omega$,
where $\nabla$ and $\brot$ are understood as a row vector, and $\ddiv$ and $\rot$ are taken on each row. Similarly, we also introduce another vector function $\bfpsi = [\psi_R, \psi_I]^T$ satisfying
\begin{equation}
\label{lem_embd_1_eq1_2}
-\ddiv(\bfB^{-1} \nabla \bfpsi ) =  -\ddiv( \bfB^{-1} \, [\bfu_R, \bfu_I]^T ), ~~~~ [\bfpsi]_{\Gamma} = \mathbf{0}, ~~~~ [ (\bfB^{-1} \nabla \bfpsi)\bfn ]_{\Gamma} = 0,
\end{equation}
subject to $\nabla\bfpsi\cdot\bfn = \mathbf{ 0}$ on $\partial \Omega$.
We claim that the following identity of matrices holds:
\begin{equation}
\label{lem_embd_1_eq2}
[\bfu_R ,\bfu_I ]^T = \bfB \, \brot \bfphi + \nabla \bfpsi.
\end{equation}
To see this, we denote $\bfTheta = [\bfu_R ,\bfu_I ]^T - \bfB \, \brot \bfphi - \nabla \bfpsi$ and note that $\rot( \bfTheta )= \ddiv(\bfB^{-1} \bfTheta ) = \mathbf{ 0}$ according to the equations \eqref{lem_embd_1_eq1_1} and \eqref{lem_embd_1_eq1_2}. Hence, by the exact sequence, there exist a column vector function $\bfvarphi$ such that $\bfTheta = \nabla \bfvarphi$, and thus we conclude the system for $\bfvarphi$:
\begin{equation}
\label{lem_embd_1_eq2_1}
\ddiv(\bfB^{-1} \nabla\bfvarphi ) = \mathbf{ 0}, ~~~ \nabla\bfpsi\cdot\bfn = \mathbf{ 0} ~~  \text{on} ~ \partial \Omega,
\end{equation}
which leads to $\bfTheta =  \nabla\bfvarphi =  \mathbf{ 0}$, based on the positivity of $\bfB^{-1}$. By the elliptic regularity, there holds $\bfphi\in \bfH^{1+\theta}(\Omega^{\pm})$ and $\bfpsi\in \bfH^{1+\theta}(\Omega^{\pm})$ with 
$$
\| \bfphi \|_{1+\theta,\Omega^{\pm}}  \lesssim  \|  \rot( [\bfu_R ,\bfu_I ]^T ) \|_{0,\Omega} ~~~ \text{and} ~~~ \| \bfpsi \|_{1+\theta,\Omega^{\pm}}  \lesssim \| \ddiv( \bfB^{-1} \, [\bfu_R, \bfu_I]^T ) \|_{0,\Omega}.
$$ 
Hence, we have $\bfu_R,~\bfu_I\in \bfH^{\theta}(\Omega^{\pm})$. At last, by introducing $\phi = \phi_R + i \phi_I$ and $\psi = \psi_R + i \psi_I$, we note that \eqref{lem_embd_1_eq2} can be equivalently rewritten into
\begin{equation}
\label{lem_embd_1_eq2_2}
\bfu = \beta^{-1} \brot \phi + \nabla \psi, 
\end{equation}
Therefore, it is not hard to see that $ \| \ddiv( \bfB^{-1} \, [\bfu_R, \bfu_I]^T ) \|_{0,\Omega} = \| \ddiv(\beta \bfu ) \|_{0,\Omega}$ which then yields $\bfH_0(\rot;\Omega)\cap \bfH(\ddiv,\beta;\Omega)\hookrightarrow \bfH^{\theta}(\Omega^{\pm})$. 

 \section{An inequality regarding fractional order Sobolev norm on irregular elements}
 \begin{lemma}
 \label{lem_geometry_patch}
 Let $S$ be a square with an edge length $h$, and let $\widetilde{S}$ its rectangular subregion which has one side overlapping with $S$ and the corresponding supporting height is $l$. 
 Then, for every $u\in H^{s}(S)$, $s>1/2$, there holds
 \begin{equation}
\label{lem_geometry_patch_eq0}
\int_{\widetilde{S}} \int_{\widetilde{S}} (u(\bfa) - u(\bfb) )^2 \dd \bfa \dd \bfb \le h^2 l^{2s}  |u |^2_{s,S}.
\end{equation}
 \end{lemma}
 \begin{proof}
  Let $\bfa=(a_1,a_2)$ and $\bfb=(b_1,b_2)$. 
 Without loss of generality, we construct a coordinate system in which the common edge of $S$ and $\widetilde{S}$ is the $x_1$ axis and the perpendicular edge is the $x_2$ axis. Then, we have
 \begin{equation}
 \begin{split}
\label{lem_geometry_patch_eq1}
\int_{\widetilde{S}} \int_{\widetilde{S}} (u(\bfa) - u(\bfb) )^2 \dd \bfa \dd \bfb & = \int_0^l \int_0^l \int_0^h \int_0^h (u(a_1,a_2) - u(b_1,b_2) )^2  \dd a_1 \dd b_1 \dd a_2 \dd b_2 \\
& \lesssim \underbrace{ \int_0^l \int_0^l \int_0^h \int_0^h (u(a_1,a_2) - u(a_1,b_2) )^2  \dd a_1 \dd b_1  \dd a_2 \dd b_2}_{(I)} \\
& + \underbrace{ \int_0^l \int_0^l \int_0^h \int_0^h (u(a_1, b_2) - u(b_1,b_2) )^2  \dd a_1 \dd b_1 \dd a_2 \dd b_2 }_{(II)}.
\end{split}
\end{equation}
Let $e_{a_1}$ be the edge perpendicular to $e$ with the $x_1$ coordinate $a_1$. We know $u \in H^{s-1/2}(e_{a_1})$, and thus
\begin{equation}
\begin{split}
\label{lem_geometry_patch_eq2}
(I) & = h \int_0^h \int_0^l \int_0^l  (u(a_1,a_2) - u(a_1,b_2) )^2 \dd a_2 \dd b_2  \dd a_1  \\
& \le h l^{2s}  \int_0^h \int_0^l \int_0^l  (u(a_1,a_2) - u(a_1,b_2) )^2/ |a_2-b_2|^{1+2(s-1/2)} \dd a_2 \dd b_2  \dd a_1 \\
& \le h l^{2s} \int_0^h |u(a_1,\cdot)|^2_{s-1/2,e_{a_1}} \dd a_1 \le h^2 l^{2s}  |u |^2_{s,S},
\end{split}
\end{equation}
where we have used the definition of fractional order Sobolev normal and the trace inequality to one side of $S$ which is shape regular.
As for $(II)$, we employ the similar argument:
\begin{equation}
\begin{split}
\label{lem_geometry_patch_eq3}
(II) &= l \int_0^l  \int_0^h \int_0^h (u(a_1, b_2) - u(b_1,b_2) )^2  \dd a_1 \dd b_1  \dd b_2 \\
& \le l h^{2s} \int_0^l  \int_0^h \int_0^h (u(a_1, b_2) - u(b_1,b_2) )^2/ |a_1 - b_2|^{1+2(s-1/2)}  \dd a_1 \dd b_1  \dd b_2 \le l^2 h^{2s}  |u |^2_{s,S}.
\end{split}
\end{equation}
Combining \eqref{lem_geometry_patch_eq1} and \eqref{lem_geometry_patch_eq2} yields the desired result.
 \end{proof}

 

 
 

\bibliographystyle{siamplain}
\bibliography{RuchiBib.bib}

\end{document}